\documentclass[11pt, reqno]{amsart}
\usepackage{amsfonts,amssymb,amscd,amstext,mathrsfs,color}
\usepackage[dvipsnames]{xcolor}
\usepackage{graphicx}
\usepackage{tikz}
\usepackage{hyperref}
\input xy
\xyoption{all}

\numberwithin{equation}{section}

\newtheorem{theorem}{Theorem}[section]
\newtheorem{corollary}[theorem]{Corollary}
\newtheorem{lemma}[theorem]{Lemma}
\newtheorem{proposition}[theorem]{Proposition}

\theoremstyle{definition}
\newtheorem{remark}[theorem]{Remark}

\newtheorem{definition}[theorem]{Definition}
\newtheorem{example}[theorem]{Example}

\newcommand{\C}{\mathbb{C}}
\newcommand{\D}{\mathbb{D}}
\newcommand{\B}{\mathbb{B}}
\newcommand{\N}{\mathbb{N}}

\newcommand{\R}{\mathbb{R}}

\renewcommand{\H}{\mathbb{H}}

\renewcommand{\Re}{{\operatorname{Re}\,}}

\begin{document}

\title{The horofunction boundary of\\ a Gromov hyperbolic space}

\author{Leandro Arosio$^*$, Matteo Fiacchi$^{\dag*}$, S\'ebastien Gontard$^{\dag*}$  and Lorenzo Guerini}
\thanks{$\dag$  Partially supported by PRIN Real and Complex Manifolds: Topology,
 Geometry and holomorphic dynamics n. 2017JZ2SW5}
 \thanks{$*$ Partially supported by MIUR
 Excellence Department Project awarded to the Department of
 Mathematics, University of Rome Tor Vergata, CUP E83C18000100006.}

\begin{abstract}
We highlight a condition, the approaching geodesics property, on a proper geodesic Gromov hyperbolic metric space, which implies that the horofunction compactification is topologically equivalent to the Gromov compactification. It is known that this equivalence does not hold in general. We prove using rescaling techniques that the approaching geodesics property is satisfied by bounded strongly pseudoconvex domains of $\C^q$ endowed with the Kobayashi metric. We also show that 
 bounded convex domains of $\C^q$ with boundary of finite type in the sense of D'Angelo satisfy a  weaker property, which still implies the equivalence of the two said compactifications.
 As a consequence  we prove that on those domains  big and small horospheres as defined by Abate in \cite{Ab1988} coincide.
  Finally we generalize the classical Julia's lemma, giving applications to the dynamics of non-expanding maps.
\end{abstract}

\maketitle
\tableofcontents

\section{Introduction}
The concept of \textit{horosphere} plays an important role in geometric function theory and in the dynamics of bounded domains of $\C^q$. Cornerstone results like the  Julia's lemma  or the Denjoy--Wolff theorem are based on horospheres \cite{Ab1988,Ab1990,AbBook, AbRa, Bud}. Recently, horospheres  have been used to prove the existence of backward orbits converging to a boundary repelling fixed point in  bounded strongly convex domains \cite{AAG,AG}, and  to prove the Muir-Suffridge conjecture on the continuous extension to the boundary of biholomorphisms from the ball to a bounded convex domain \cite{braccigaussier, braccigaussier2}. In bounded strongly convex domains horospheres are sublevel sets of a pluricomplex Poisson kernel \cite{poisson1,poisson2}.

Horospheres of the unit ball $\B^q\subseteq \C^q$ can be defined in terms of 
the Kobayashi metric $d_{\B^q}$ of $\mathbb{B}^q$. The horosphere  centered at $\xi\in \partial \B^q$, of radius $R>0$, with  base point $p\in \B^q$ is the domain
$$E_p(\xi,R):=\left\lbrace z\in \B^q\colon \lim_{w\to \xi}d_{\B^q}(z,w)-d_{\B^q}(w,p)<\log R\right\rbrace.$$ 
 The definition given by some authors, see e.g. \cite{AbBook}, shows a factor {1/2} which is due to a different normalization of the Kobayashi metric.
In order to extend this definition to a boundary point $\xi$ of a given bounded domain $D\subseteq \C^q$, endowed with the Kobayashi metric $d_D$, one needs to ensure that the limit
\begin{equation}\label{abatelimit}
\lim_{w\to \xi}d_{D}(z,w)-d_{D}(w,p)
\end{equation}
exists. This is not the case in general, for example the limit \eqref{abatelimit} does not exist if $\xi$ is a point in the distinguished boundary of the  bidisc $\D^2$.
In 1988 Abate defined  \cite{Ab1988} the \textit{big horosphere and the small horosphere} centered at $\xi$ of radius $R>0$ with base point $p$ respectively as
\begin{align}
\label{bigsmallhor}
\begin{split}
E^{b}_p(\xi,R)&:=\left\lbrace z\in D\colon \liminf_{w\to \xi}d_{D}(z,w)-d_{D}(w,p)<\log R\right\rbrace,\\
E^{s}_p(\xi,R)&:=\left\lbrace z\in D\colon \limsup_{w\to \xi}d_{D}(z,w)-d_{D}(w,p)<\log R\right\rbrace.
\end{split}
\end{align}

Later, using Lempert's theory of complex geodesics \cite{Lempert}, he proved that if $D$ is a bounded strongly convex domain with boundary of class $C^3$, then the limit \eqref{abatelimit} exists and thus big and small horospheres coincide (see \cite{Ab1990}). 
The question whether the same holds for strongly pseudoconvex domains (or more generally for smoothly bounded pseudoconvex domains of finite type) has been open since,   one reason being that Lempert's theory  is not available for such domains.

 In this paper we avoid this issue following a  different approach by means of the horofunction boundary $\partial_HD$ and the corresponding compactification  $\overline{D}^H$ (Definition \ref{horofunctionboundary}).
These notions were introduced by Gromov in 1981 \cite{Gromov}, and  have been recently studied in  several research areas including geometry of  bounded convex domains of $\R^q$  with the Hilbert metric \cite{walsh}, Teichm\"uller spaces \cite{alessandrini},   and random walks on groups of isometries \cite{karlssonannals,tiozzo}.
The existence of the limit \eqref{abatelimit} is a consequence of the existence of a  continuous map $\Phi\colon \overline D\to \overline{D}^H$ extending the identity.
From this perspective, horospheres centered at $\xi\in \partial D$ coincide with sublevel sets of a horofunction centered at $\Phi(\xi)\in \partial_HD$.

If $D$ is strongly pseudoconvex, then Balogh and Bonk \cite{BaBo} proved that $D$ is Gromov hyperbolic and that there exists an  homeomorphism $\overline D \rightarrow \overline{D}^G$ extending the identity map. 
Hence the existence of a continuous map   $\overline{D}^G\to\overline{D}^H$ extending the identity  would  imply the existence of the limit \eqref{abatelimit}. On the other hand, there exist  proper geodesic Gromov hyperbolic metric spaces for which such  map cannot exist (see  Example \ref{wwexample}).
In Section \ref{approachingsection} we introduce a sufficient condition for the existence of such a map: we say that a proper geodesic  metric space $(X,d)$ has  \textit{approaching geodesics} if any two asymptotic geodesic rays $\gamma,\sigma$ are \textit{strongly asymptotic}, meaning that there exist $T \in \mathbb{R}$ such that 
$$\displaystyle \lim_{t\to+\infty}d(\gamma(t),\sigma(t+T))= 0.$$ 
This condition holds for example for $CAT(-1)$ spaces, 0-hyperbolic spaces, or Gromov hyperbolic bounded convex  domains of $\R^n$ endowed with the Hilbert metric.   Our first result, proved in Theorem \ref{abstract}  is the following:
\begin{theorem}\label{intro}
Let $(X,d)$ be a proper geodesic Gromov hyperbolic metric space with approaching geodesics. Then the Busemann map induces a  homeomorphism extending the identity $ \overline{X}^G\to \overline{X}^H.$
\end{theorem}

In Sections  \ref{sec:strong} and \ref{sec:finite}  we study whether bounded strongly pseudoconvex domains and convex domains of finite D'Angelo type have approaching geodesics. Our approach is similar in both cases.
The idea is the following: assume  by contradiction that there exists two asymptotic  geodesic rays $\gamma,\sigma$ which are not strongly asymptotic. 
Then rescale along a sequence of points in $\gamma$, to obtain in the limit two different geodesic lines of some \textit{model domain}, connecting the same two points in the boundary.
This leads to a contradiction if we know that  the model domain admits a unique geodesic line joining any two points at the boundary.

In Section \ref{sec:strong} we study bounded strongly pseudoconvex domains in $\C^q$. 
We use the squeezing function to rescale the domain, and the corresponding model domain is the unit ball $\B^q$, therefore strongly pseudoconvex domains have approaching geodesics.
Since our arguments only rely on the boundary behavior of the squeezing function, we also obtain that bounded convex domains satisfying $\lim_{z\to\partial D}s_D(z)=1$ have approaching geodesics.

In Section \ref{sec:finite} we study bounded convex domains of finite type in the sense of D'Angelo. Such spaces are Gromov hyperbolic according to a recent result of Zimmer \cite{Zim1}. 
In this case,  after rescaling in the normal direction, one obtains  models of the form
$$\left\lbrace (z,w)\in \C^{q-1}\times \C\,:\,{\rm Re\,} w> H(z)\right\rbrace,$$ 
where $H\colon \C^{q-1}\to \R$ is a convex non-degenerate non-negative weighted homogeneous polynomial. 
We do not know whether 
the straight half-line $\{z=0,{\rm Im}\, w=0\}$ is the unique geodesic line connecting $0$ to $\infty$. 
Hence, we assume that the two asymptotic geodesic rays lie in two complex geodesics $\gamma,\sigma$. We prove that both $\gamma$ and $\sigma$ rescale into the same complex geodesic $\{z=0\}$, and 
this is possible only when the asymptotic rays lying in $\gamma$ and $\sigma$ are strongly asymptotic. It turns out that this weaker form of the approaching geodesics property is enough to conclude that the Gromov and horofunction compactifications are equivalent (see Theorem \ref{weaknotion}).

Our main results can thus be stated as follows.
\begin{theorem}\label{intro2}
Let $(D,d_D)$ be either
\begin{enumerate}
\item a bounded strongly pseudoconvex domain with $C^2$ boundary in $\C^q$, 
\item a bounded convex domain such that $\displaystyle \lim_{z\to \partial D} s_D(z)=1$,
\item a bounded convex domain in $\C^q$ with $C^\infty$ boundary,  of finite type in the sense of D'Angelo.
\end{enumerate}
Then any two of the compactifications $\overline{D}^G,\overline{D}^H, \overline D$ admits homeomorphisms extending the identity. 
Hence for all $\xi\in \partial D$ and $p\in D$ the limit \eqref{abatelimit} exists, and thus  big and small horospheres coincide, that is 
$$E^b_p(\xi,R)=E^s_p(\xi,R),\quad \forall \,R>0.$$
\end{theorem}
Theorem \ref{intro2} is a consequence of Corollary  \ref{ohyeah1.5.1} and Theorem \ref{mainfinitetype}. We remark that the equivalence between Gromov and Euclidean compactifications  had already been proven in \cite{BaBo,Zim1} for cases (1) and (3), and in case (2) this equivalence follows immediately from results in \cite{BrGaZi,Zim6}. 
Notice also that since the horofunction and  Gromov compactifications are metric invariants, the  equivalence between 
$\overline{D}^G$ and $\overline{D}^H$ also holds for every domain of $\C^q$ which is biholomorphic to the ones in (1),(2), and (3).


In the last section we apply our results to the study of the  forward and backward dynamics of non-expanding self-maps of a proper geodesic Gromov metric space, obtaining in particular a generalization of the classical Julia's lemma when the horofunction and  Gromov compactifications are equivalent.

\medskip {\bf Acknowledgements.} We want to thank  Andrew Zimmer for useful comments and remarks.

\medskip{\bf Competing interests declaration}
On behalf of all authors, the corresponding author states that there is no conflict of interest.

\section{Background}
\label{sec:background}

\begin{definition}\label{defcomp}
Let $X$ be a topological space. A \textit{compactification} of $X$ is a couple $(i,Y)$ where $Y$ is a compact topological space and $i\colon X\to Y$ is a topological embedding such that $\overline{i\left(X\right)}=Y$. A compactification is called \textit{Hausdorff} if $Y$ is Hausdorff. Define a relation on the set of compactifications of $X$ setting $(i_1,Y_1)\geq (i_2,Y_2)$ if and only if there exists a continuous map $h\colon Y_1\to Y_2$ such that $h\circ i_1=i_2$. Notice that $h$ is necessarily surjective.  If $h$ is a homeomorphism, the two compactifications are \textit{topologically equivalent}.
\end{definition}
For the sake of simplicity the compactification $(Y,i)$ will be denoted as $Y$.

\begin{remark}
Assume that  $Y_1$ and $Y_2$ are two Hausdorff compactifications of $X$. If $Y_1\geq Y_2$ and $Y_2\geq Y_1$, then  $Y_1$ and $Y_2$ are topologically equivalent.
\end{remark}
We refer the reader to \cite{kelly} for more details about compactifications of a topological space.

 Our main result relates two natural compactifications: the horofunction compactification and the Gromov compactification.  We now recall their definitions.

\begin{definition}
Let $(X,d)$ be a metric space.
 A {\sl  geodesic} is a map $\gamma$ from an interval $I\subset \R$ to $X$ which is an isometry with respect to the Euclidean distance on $I$ and the distance on $X$, that is for all $s,t\in I$, $$d(\gamma(s),\gamma(t))=|t-s|.$$
If the interval is closed and bounded (resp. $[0,+\infty)$,  $(-\infty,+\infty)$) we call $\gamma$ a {\sl geodesic segment}  (resp. {\sl geodesic ray}, {\sl geodesic line}).
 A metric space $(X,d)$ is \textit{geodesic} if any two points are joined by a geodesic segment.
 A metric space $(X,d)$ is \textit{proper} if and only if every closed ball is compact.
\end{definition}
\begin{remark}
A proper metric space is locally compact, hemicompact, and second countable.   Notice also that if a topological space $X$ is locally compact, then for any compactification $Y$ of $X$ we have that $Y\setminus X$ is compact.
\end{remark}

Given a proper metric space $(X,d)$, we write $C(X)$ for the space of continuous real functions on $X$. 
Endow $C(X)$ with the topology given by uniform convergence on compact subsets. 
\begin{definition}[Horofunction compactification]\label{horofunctionboundary}
Let $(X,d)$ be a proper metric space. 
Let $C_*(X)$ be the quotient of $C(X)$ by the subspace of constant functions. 
Given $f\in C(X)$,  we denote its equivalence class by $[f] \in C_*(X)$. 

Consider the embedding 
$$i_H\colon X\longrightarrow C_*(X)$$ which sends a point $x\in X$ to  the equivalence class of the function  $d_x\colon y\mapsto d(x,y).$ 
The \textit{horofunction compactification} $ \overline{X}^H$ of $X$ is the closure of $i_H(X)$ in  $C_*(X)$. 
The \textit{horofunction boundary} of $X$ is the (compact) set
$$\partial_HX:=   \overline{X}^H\setminus i_H(X).$$
Let $a \in \partial_H X$. An \textit{horofunction} centered at $a\in \partial_H X$ is an element $h\in C\left(X\right)$ satisfying $[h]=a$. 
For every $p\in X$, the unique horofunction centered at $a$ and vanishing at $p$ is denoted by $ h_{a,p}$.
The \textit{horosphere (or horoball)}  centered at $a$ of radius $R>0$ and with base point $p\in X$ is the level set
\begin{equation}
\label{eq:horo}
E_p(a,R):=\left\lbrace h_{a,p}<\log R\right\rbrace\subseteq X.
\end{equation}
\end{definition}

\begin{remark}
Denote $C_{p}(X)$ the subspace of $C(X)$ of all functions vanishing at a base point $p\in X$. 
The linear map $[f]\longmapsto f-f(p)$ induces a homeomorphism between $C_*(X)$ and $C_{p}(X)$.
A sequence $ \left([f_n]\right)$ converges to $[f]$ in $C_*(X)$ if and only if $\displaystyle \left(f_n-f_n\left(p\right)\right)$ converges to $\left(f-f\left(p\right)\right)$ in $C\left(X\right)$. 
Definition \ref{horofunctionboundary} has the advantage to show that the horofunction compactification is uniquely defined, and does not depend on the choice of the base point. In practice, it is more convenient to work with $C_p(X)$ instead of $C_*(X)$. 
\end{remark}

Busemann points are classical examples of horofunctions. Given a metric space $\left(X,d\right)$ we denote by ${\mathscr{R}(X)}$ the set of geodesic rays of $X$. For every geodesic ray $\gamma \in \mathscr{R}\left(X\right)$ the family $$\Big(x\mapsto d\big(x,\gamma(t)\big)-d\big(\gamma(t),\gamma(0)\big)\Big)_{t\geq 0}$$ is a pointwise non-decreasing, locally bounded from above family of $1$-Lipschitz maps. Thus it converges (uniformly on compact subsets) as $t\to +\infty$. This observation makes the following definitions meaningful:
\begin{definition}
Let $(X,d)$ be a proper metric space and $\gamma\colon \R_{\geq 0}\to X$ be a geodesic ray. The \textit{Busemann function} $B_\gamma\colon X\times X\rightarrow \R$ associated with $\gamma$ is defined as
$$B_\gamma(x,y):=\displaystyle \lim_{t\to +\infty} d(x,\gamma(t))-d(\gamma(t),y).$$
For all $y\in X$, the function $x\mapsto B_\gamma(x,y)$ is a horofunction, and its class  $[B_\gamma]\in \partial_HX$ does not depend on $y\in X$.

The \textit{Busemann map} $B\colon\mathscr R(X)\rightarrow \partial_HX$ is defined as $B(\gamma)=[B_\gamma]$.
Elements of $B\left(\mathscr{R}(X)\right)$ are called Busemann points.
\end{definition}

Next we turn our attention to the Gromov compactification.

\begin{definition}
Let $(X,d)$ be a proper metric space.
Two geodesic rays $\gamma,\sigma\in \mathscr{R}(X)$ are \textit{asymptotic} if $t\mapsto d(\gamma(t),\sigma(t))$ is  bounded.
\end{definition}

We recall that a complete length space is proper and geodesic. This applies in particular when $D\subseteq \mathbb C^{q}$  is a complex hyperbolic domain endowed with the Kobayashi metric.

\begin{definition}[Gromov compactification]\label{visual}
Let $(X,d)$ be a proper geodesic Gromov hyperbolic metric space. 
On $\mathscr{R}(X)$, the relation 
\[\gamma \sim_r \sigma \iff \text{$\gamma$ and $\sigma$ are asymptotic}\]
is an equivalence relation. 
The \textit{Gromov boundary}  of $X$ is defined as $\partial_GX:=\mathscr{R}(X)/_{\sim_r}.$
The \textit{Gromov compactification}  of $X$ is the set $\overline X^G:=X\sqcup \partial_GX$ endowed with a  metrizable topology \cite[Chapter III.H]{BH} which makes it a compactification in the sense of Definition \ref{defcomp}. We now describe which sequences converge to a point in the boundary.
Fix a base point $p$. A sequence $(x_n)$ in $\overline X^G$ converges to a point $\xi\in \partial_GX$ if and only if for any subsequence of $(\gamma_n)$ that converges uniformly on compact subsets, its limit belongs to the equivalence class  $\xi$. Here $\gamma_n$ is defined as follows:
if $x_n\in X$, then  $\gamma_n\colon [0,T_n]\to X$ is a geodesic segment connecting $p$ to $x_n$, while if $x_n\in \partial_GX$, then 
$\gamma_n\colon \R_{\geq 0}\to X$ is a geodesic ray representing  $x_n$ with $\gamma_n(0)=p$.
\end{definition}
\begin{remark}
It is possible to show that the choices of $p$ and $(\gamma_n)$ in the definition are not relevant.
Furthermore since $(X,d)$ is proper, by Ascoli--Arzel\`a Theorem, every subsequence of $(\gamma_n)$ admits a subsequence that converges uniformly on compact subsets to a geodesic ray.
In particular we can always find convergent subsequences of $(\gamma_n)$ as in the definition.
\end{remark}
\begin{remark}
\label{infinitysequences}
There is an alternative description of the Gromov boundary $\partial_GX$.
Fix $p\in X$.
We say that a sequence $\left(x_n\right)$ \textit{goes to infinity in the Gromov sense} if  
\[\left(x_n\vert x_m\right)_p:=\frac{1}{2}\left(d\left(x_n,p\right)+d\left(x_m,p\right)-d\left(x_n,x_m\right)\right)\underset{m,n \to +\infty}{\longrightarrow}+\infty.\]
We denote by $\mathcal{I}^G$ the set of sequences going to infinity in the Gromov sense. Since $X$ is Gromov hyperbolic,  the relation
\[\left(x_n\right)\sim_s \left(y_n\right) \iff \left(x_n\vert y_n\right)_p\underset{n \to +\infty}{\longrightarrow}+\infty\]
is an equivalence relation on $\mathcal{I}^G$. 
Both definitions do not depend on the point $p$ chosen.
The set $\mathcal{I}^G/{\sim_s}$ admits a canonical bijection with $\partial_GX$, and so we can identify them. 
Furthermore a sequence $(x_n)$ in $X$ converges to $\xi\in \partial_G X$ if and only if the sequence $(x_n)$ goes to infinity in the Gromov sense and $[(x_n)]=\xi$.	
\end{remark}

It is natural to ask whether there is a relation between these two  compactifications. Webster and Winchester prove in \cite{WWpacific} that if $(x_n), (y_n)$ are sequences in $X$ converging to the same point in $\partial_HX$, then 
$(x_n)\sim_s(y_n)$. It easily follows (see \cite[Proposition 4.6]{WWpacific})  that
\begin{equation}\label{WW}
\overline{X}^H\geq  \overline{X}^G.
\end{equation}
If the space $X$ is  $CAT(0)$ or $0$-hyperbolic, then  the two compactifications are actually topologically equivalent (see \cite[Theorem 8.13]{BH} and \cite[Theorem 4.7]{WWpacific}). It follows by a direct computation that this is the case also for Gromov hyperbolic  bounded convex domains of $\R^q$ endowed with the Hilbert metric.

In general the two compactifications are not topologically equivalent, as the following example adapted from \cite{WWtrans} shows.
\begin{example}\label{wwexample}
Let $X$ be the following closed set of $\R^2$:
$$X=\big(\mathbb N\times(-1,1)\big)\cup\big(\R_{\geq0}\times\{-1,1\}\big). $$
\begin{figure}[ht]
\begin{tikzpicture}[scale=2,  mydot/.style={circle, fill=white, draw, outer sep=0pt,inner sep=1.5pt}]
    \draw[line width=.5mm] (0,1)--(4.7,1);
    \draw[line width=.5mm] (0,-1)--(4.7,-1);
    \draw[line width=.5mm] (0,1)--(0,-1);
    \draw[line width=.5mm] (1,1)--(1,-1);
    \draw[line width=.5mm] (2,1)--(2,-1);
    \draw[line width=.5mm] (3,1)--(3,-1);
    \draw[line width=.5mm] (4,1)--(4,-1);

    \node at (5,1) {$\dots$};
    \node at (5,0) {$\dots$};
    \node at (5,-1) {$\dots$};

    \node[mydot]  at (0,1) {};
    \node[below right] at (0,1) {$(0,1)$};
    \node[mydot]  at (1,1) {};
    \node[below right] at (1,1) {$(1,1)$};
    \node[mydot]  at (2,1) {};
    \node[below right] at (2,1) {$(2,1)$};
    \node[mydot]  at (3,1) {};
    \node[below right] at (3,1) {$(3,1)$};
    \node[mydot]  at (4,1) {};
    \node[below right] at (4,1) {$(4,1)$};

    \node[mydot]  at (0,0) {};
    \node[below right] at (0,0) {$(0,0)$};
    \node[mydot]  at (1,0) {};
    \node[below right] at (1,0) {$(1,0)$};
    \node[mydot]  at (2,0) {};
    \node[below right] at (2,0) {$(2,0)$};
    \node[mydot]  at (3,0) {};
    \node[below right] at (3,0) {$(3,0)$};
    \node[mydot]  at (4,0) {};
    \node[below right] at (4,0) {$(4,0)$};

    \node[mydot]  at (0,-1) {};
    \node[below right] at (0,-1) {$(0,-1)$};
    \node[mydot]  at (1,-1) {};
    \node[below right] at (1,-1) {$(1,-1)$};
    \node[mydot]  at (2,-1) {};
    \node[below right] at (2,-1) {$(2,-1)$};
    \node[mydot]  at (3,-1) {};
    \node[below right] at (3,-1) {$(3,-1)$};
    \node[mydot]  at (4,-1) {};
    \node[below right] at (4,-1) {$(4,-1)$};
\end{tikzpicture}
\caption{The set $X$}

\end{figure}
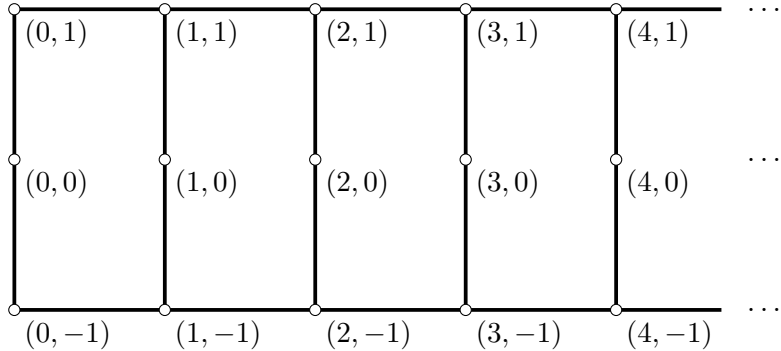

We define a distance on $X$ as the induced length distance of $\R^2$, i.e.
$$d(x,y):=\inf \{\mbox{length}(\gamma)\,|\, \gamma:[0,1]\longrightarrow X,\, \text{piecewise} \,\,C^1,\, \gamma(0)=x,\gamma(1)=y \}, $$
which makes it a proper geodesic Gromov hyperbolic metric space.
Given $|a_1-a_2|>1$, we have 
$$d((a_1,b_1),(a_2,b_2))=2+|a_1-a_2|-|b_1+b_2|.$$ 
If $\gamma:\R_{\geq0}\longrightarrow X$ is a geodesic ray then there exists $T,a\geq0$, and $\varepsilon\in\{-1,1\}$ such that $$\gamma(t)=(t+a,\varepsilon), \qquad \forall t\geq T.$$ It follows that
 every two geodesic rays are asymptotic, and  therefore $\partial_GX$ consists of a single point. We now show that the horofunction boundary $\partial_HX$ is homeomorphic to a segment.
Fix the base point $(0,0)$. Define
$$h_{(\alpha,\beta)}(a,b):=d((a,b),(\alpha,\beta))-d((\alpha,\beta),(0,0)), \qquad (\alpha,\beta),(a,b)\in X. $$
 For  $\alpha>a+1$ we have
$$h_{(\alpha,\beta)}(a,b)=-a+|\beta|-|b+\beta|. $$
A sequence $((\alpha_n,\beta_n))$ is compactly divergent if and only if $\alpha_n\rightarrow +\infty$.
Furthermore the function $h_{(\alpha_n,\beta_n)}$ converges to  a horofunction in $C(X)$ if and only if $\beta_n\to\widehat{\beta}\in[-1,1]$.
The horofunction $\widehat{h}$ is then given by $\widehat{h}(a,b)=-a+|\widehat{\beta}|-|b+\widehat{\beta}|$, so it is easy to see that $\partial_H X$ is homeomorphic to $[-1,1]$.
In this example $-1,1\in\partial_H X$ are the only Busemann points. 

\end{example}

Notice that the two geodesic rays $\gamma(t)=(t,1)$ and $\sigma(t)=(t,-1)$ are asymptotic, but  $\inf_{s\geq 0}d( \gamma(t),\sigma(s))$ does not converge to 0 as $t\to+\infty$. 
In the next section  we prove that the compactifications $ \overline{X}^H$ and $\overline{X}^G$ are topologically equivalent, provided that this behavior does not occur.

\section{The approaching geodesics property}\label{approachingsection}
\begin{definition}
Let $(X,d)$ be a proper geodesic metric space.
 Two geodesic rays $\gamma,\sigma\in \mathscr{R}(X)$ are \textit{strongly asymptotic} if there exist $T\in \R$  such that
$$\displaystyle \lim_{t\to +\infty}d(\gamma(t), \sigma(t+T))=0.$$ We say that a proper geodesic   metric space $X$ has \textit{approaching geodesics} if asymptotic rays are strongly asymptotic.
\end{definition}
Clearly strongly asymptotic rays are asymptotic.
\begin{remark}
 Notice  that
$CAT(-1)$ spaces \cite[Proposition 1.4]{CharTsa} and $0$-hyperbolic spaces \cite[Lemma 3.3]{BH} have approaching geodesics.   It is also easy to see that  Gromov hyperbolic bounded convex  domains of $\R^n$ endowed with the Hilbert metric have approaching geodesics.
\end{remark}

\begin{proposition}\label{nondip}
Let $(X,d)$ be a proper geodesic  metric space. If two geodesic rays $\gamma,\sigma\in \mathscr{R}(X)$ are strongly asymptotic, then  $B_\gamma=B_\sigma.$
In particular $[B_\gamma]=[B_\sigma].$
\end{proposition}
\begin{proof}
Let $T \in \R$ such that $\displaystyle \lim_{t\to +\infty} d(\gamma(t), \sigma(t+T))=0$. Let $x,p\in X$. Then
$$\bigg\lvert \underset{\underset{t \to +\infty}{\longrightarrow}B_\gamma(x,p)}{\underbrace {d(x,\gamma(t))-d(\gamma(t),p)}}-(\underset{\underset{t \to +\infty}{\longrightarrow}B_\sigma(x,p)}{\underbrace {d(x,\sigma(t+T))-d(\sigma(t+T),p))}}\bigg\rvert\leq 2d(\gamma(t),\sigma(t+T)).$$
We obtain the result by letting $t\to +\infty$.
\end{proof}


\begin{corollary}
Let $(X,d)$ be a proper geodesic Gromov hyperbolic metric space with  approaching geodesics.
Then the Busemann map induces a well-defined map
$$ \Psi\colon \partial_GX\to \partial_HX, \quad \left[\gamma\right]\mapsto [B_\gamma].$$
\end{corollary}

\begin{theorem}\label{abstract}
Let $(X,d)$ be a proper geodesic Gromov hyperbolic metric space with  approaching geodesics. Then the map 
$\Psi\colon  \overline{X}^G\to  \overline{X}^H$ defined by
$$\Psi(\xi)=
\begin{cases}
[d(\,\cdot\,,\xi)], & {\rm if}\ \xi \in X,\\
[B_\gamma],& {\rm if}\ \xi=[\gamma]\in \partial_GX
 \end{cases}
 $$
 is a homeomorphism, and thus $ \overline{X}^G$ and $ \overline{X}^H$ are topologically equivalent.
\end{theorem}

\begin{proof}
By \eqref{WW} we only need to prove that $\Psi$ is continuous and hence $ \overline{X}^G\geq  \overline{X}^H$.  Let $(w_n)$ be a sequence of points in $X$ converging to $\xi\in\partial_GX$. We want to prove that $(\Psi(w_n))$ converges to $\Psi(\xi)$. Given a base point $p\in X$,  this is equivalent to proving that the sequence of functions
$(d(\,\cdot\,,w_n)-d(w_n,p))$ converges uniformly on compact subsets to $B_{\gamma}(\,\cdot\,,p),$ where $\gamma$ is any geodesic ray representing $\xi$ (by Proposition \ref{nondip} the function  $B_{\gamma}(\,\cdot\,,p)$ does not depend on the choice of $\gamma$).

The functions $d(\,\cdot\,,w_n)-d(w_n,p)$ are $1$-Lipschitz, thus by Ascoli--Arzel\`a Theorem it is enough to prove pointwise convergence to $B_{\gamma}(\,\cdot\,,p)$. 
Given $z\in X$ we have $|d(z,w)-d(w,p)|\le d(z,p)$. Consequently the sequence $\left(d(z,w_n)-d(w_n,p)\right)$ is bounded. 
Therefore to prove that it converges to $B_{\gamma}(z,p)$, it suffices to prove that $B_{\gamma}(z,p)$ is the unique  accumulation point of the sequence. 
Let $(w_{n_k})$ be a subsequence such that $\left(d(z,w_{n_{k}})-d(w_{n_{k}},p)\right)$ converges.
 By Definition \ref{visual} we obtain, up to extracting another subsequence, a sequence $(\gamma_k\colon [0,T_k]\to X)$ of  geodesic segments connecting $z$ to $w_{n_k}$, converging uniformly on compact subsets to a geodesic ray $\gamma$ such that $[\gamma]=\xi$.
 Similarly we obtain a sequence   $(\sigma_k\colon [0,S_k]\to X)$ of  geodesic segments connecting $p$ to $w_{n_k}$, converging uniformly on compact subsets to a geodesic ray $\sigma$  such that  $[\sigma]=\xi$.

Since $X$ has approaching geodesics, the rays $\gamma$ and $\sigma$ are strongly asymptotic. Let $T\in \R$  such that 
$$\displaystyle \lim_{t\to+\infty} d(\gamma(t), \sigma(t+T))=0.$$
Fix $t\geq 0$. For $k$ large enough we have
\begin{align*}
 d(z, w_{n_k})&=d (z, \gamma_k(t))+d(\gamma_k(t), w_{n_k}),\\[5pt]
 d(w_{n_k},p)&\leq d(w_{n_k},\gamma_k(t))+ d(\gamma_k(t),p),\\[5pt]
 d(w_{n_k},p)&= d(w_{n_k},\sigma_k(t+T))+d (\sigma_k(t+T),p)\\
&\geq  d(w_{n_k},\gamma_k(t))+d (\gamma_k(t),p)- 2 d(\gamma_k(t),\sigma_k(t+T)).
\end{align*}
It follows that 
\begin{align*}
d(z, w_{n_k})-d(w_{n_k},p)&\geq d (z,\gamma_k(t))- d(\gamma_k(t),p),\\[5pt]
d(z, w_{n_k})-d(w_{n_k},p)&\leq d (z,\gamma_k(t))- d(\gamma_k(t),p)+2 d(\gamma_k(t),\sigma_k(t+T)).
\end{align*}
Taking the limit as $k\to \infty$ we obtain 
$$B_\gamma(z,p)\leq \lim_{k\to\infty}d(z,w_{n_k})-d(w_{n_k},p)\leq  B_\gamma(z,p)+2 d(\gamma(t),\sigma(t+T)).$$
Thus by letting $t\to +\infty$ it follows $ \displaystyle \lim_{k\to\infty}d(z,w_{n_k})-d(w_{n_k},p)=  B_\gamma(z,p)$, as desired.

Assume now that  $(\xi_n)$ is a sequence in $\partial_GX$ converging to $\xi\in\partial_G X$.
We want to prove that $(\Psi(\xi_n))$ converges to $\Psi(\xi).$ Since $ \overline{X}^H$ is compact, it is enough to prove that every convergent subsequence $(\Psi(\xi_{n_k}))$ converges to $\Psi(\xi)$.

Fix a base point $p\in X$. By Definition \ref{visual} we obtain, up to extracting another subsequence,   a sequence   $(\gamma_k)$ of  geodesic rays such that $\gamma_k(0)=p$ and $[\gamma_k]=\xi_{n_k}$ converging uniformly on compact subsets to a geodesic ray $\gamma$ such that $\gamma(0)=p$ and $[\gamma]=\xi.$
To end the proof, we just need to show that  the sequence of functions $(B_{\gamma_{k}}(\,\cdot\,,p))$ converges uniformly on compact subsets to $B_{\gamma}(\,\cdot\,,p)$.

By the first part of the proof, if $(w_n)$ is a sequence of points in $X$ converging to $\xi$, then $(d(\,\cdot\,,w_n)-d(w_n,p))$ converges uniformly on compact subsets to $B_{\gamma}(\,\cdot\,,p).$ 
Let $\widetilde d$ be a metric inducing the topology of $\overline X^G$, and let $K\subseteq X$ be a compact set.
 For all $k\geq 0$ set $w_k:=\gamma_{k}(T_{k})$, where $T_{k}\geq k$ is large enough such that
\begin{equation}
\label{quellocheserve}
\widetilde d(w_k,\xi_{n_k})\le \frac{1}{k},
\end{equation}
and such that
\begin{equation}
\label{pm}
\sup_{x\in K}|d(x,w_k)-d(w_k, p)-B_{\gamma_{k}}(x,p)|\leq \frac{1}{k}.
\end{equation}

By \eqref{quellocheserve} we have that $w_k\to\xi$, hence the sequence $(d(\,\cdot\,,w_k)-d(w_k,p))$ converges to $B_\gamma(\,\cdot\,,p)$ uniformly on $K$. By \eqref{pm} the same is true for $(B_{\gamma_{k}}(z,p))$, concluding the proof.
\end{proof}

\begin{remark}

The approaching geodesics property is not necessary for the horofunction  and  Gromov compactifications to be topologically equivalent.  Indeed such equivalence holds in any proper geodesic Gromov hyperbolic $CAT(0)$ metric space, as  for example the Euclidean strip 
$\R\times [-1,1]$, which does not have approaching geodesics.
\end{remark}

\begin{corollary}
Let $(X,d)$ be a proper geodesic Gromov hyperbolic metric space with  approaching geodesics. Then all points of $\partial_HX$ are Busemann points.
\end{corollary}

In the proof of Theorem \ref{abstract} we actually proved the following result, which focuses on a point $\xi\in \partial_G X$.
\begin{corollary}
Let $(X,d)$ be a proper geodesic Gromov hyperbolic metric space and let $\xi=[\gamma]\in \partial_G X$. Assume that every two rays in the class of $\xi$ are strongly asymptotic. Then the map 
$\Psi_\xi\colon X\cup\{\xi\}\to \overline X^H$ defined by $\Psi_\xi(\xi)=[B_\gamma]$ and by
$\Psi_\xi(x)=[d(\,\cdot\,,x)]$ for all $x\in X$ is continuous at $\xi$.
In particular, $\forall z,p \in X$,
\[\lim_{w\to \xi}d\left(z,w\right)-d\left(w,p\right)=B_{\gamma}\left(z,p\right).\]
\end{corollary}

Up to minor adaptations, the proof and result of Theorem \ref{abstract} remain valid if we replace the approaching geodesics property with the following weaker property.
\begin{definition}
\label{def_weaknotion}
We say that a proper geodesic metric space $(X,d)$ satisfies the \textit{weak approaching geodesics property} relative to a subset $\mathscr{F}\subseteq \mathscr{R}\left(X\right)$ if the following two conditions hold:
\begin{enumerate}
	\item for every $p\in X$,  $\xi\in \partial_GX$, and every  sequence $(x_n)$ in $X$ converging to $\xi$ there exist a subsequence $\left(x_{n_k}\right)$ and a family of geodesics segments $\left(\gamma_k\right)$ such that $\gamma_k$ joins $p$ and $x_{n_k}$, and such that $\left(\gamma_k\right)$ converges uniformly on compact subsets to a geodesic ray in $\mathscr{F}$;
	\item asymptotic geodesic rays in $\mathscr{F}$ are strongly asymptotic.
\end{enumerate}
\end{definition}
\begin{remark} If $(X,d)$ is Gromov hyperbolic, then condition (1) implies that 
 for every $p\in X$ and $\xi\in\partial_GX$ there exists $\gamma\in\mathscr{F}$ such that $\gamma(0)=p$ and $[\gamma]=\xi$.
 \end{remark}
 \begin{theorem}\label{weaknotion}
Let $(X,d)$ be a proper geodesic Gromov hyperbolic metric space satisfying  the weak approaching geodesics property relative to a family $\mathscr{F}$ of geodesic rays. Then the map 
$\Psi\colon  \overline{X}^G\to  \overline{X}^H$ defined by
$$\Psi(\xi)=
\begin{cases}
[d(\,\cdot\,,\xi)], & {\rm if}\ \xi \in X,\\
[B_\gamma],& {\rm if}\ \xi=[\gamma]\in \partial_GX
 \end{cases}
 $$
 is a homeomorphism, and thus $ \overline{X}^G$ and $ \overline{X}^H$ are topologically equivalent.
\end{theorem}

We end this section by introducing an equivalent formulation of  strong asymptoticity which will be useful in the next sections.
\begin{remark}
Let $(X,d)$ be a metric space. Given a geodesic ray $\gamma\colon \R_{\geq 0}\to X$ then for every point $p\in X$ the infimum $\displaystyle \inf_{s\geq 0}d( p,\gamma(s))$ is attained. Indeed
$$d(p, \gamma(s))\geq d( \gamma(s),  \gamma(0))-d( \gamma(0),p)\underset{s\to +\infty}{\longrightarrow} +\infty.$$ The same holds for a geodesic line. 
\end{remark}

\begin{proposition}
\label{prop:approach}
Let $(X,d)$ be a proper metric space and $\gamma$ and $\sigma$ be two geodesic rays. 
The rays $\gamma$ and $\sigma$ are strongly asymptotic if and only if
\begin{equation}
\label{variationasymp}
\displaystyle \lim_{t\to+\infty}\inf_{s\geq 0}d( \gamma(t),\sigma(s))  = 0.
\end{equation}
\end{proposition}
\begin{proof}
If $\gamma$ and $\sigma$ are strongly asymptotic then \eqref{variationasymp} is obvious. Assume on the other hand that $\displaystyle \lim_{t\to+\infty}\inf_{s\ge 0}d( \gamma(t),\sigma(s))  = 0.$
Let $s_t\geq 0$ be such that
$$
d(\gamma(t),\sigma(s_t))=\inf_{s\geq 0}d(\gamma(t),\sigma(s)).
$$
Notice that $\displaystyle \lim_{t\to +\infty}s_t=+\infty.$

We may choose $0\le t_1\le T_1$ so that $d(\gamma(t),\sigma(s_t))\le 1$ for every $ t\ge t_1$ and $s_t\ge s_{t_1}$ for all $t\ge T_1$. 

Once $t_1$ and $T_1$ are given, we may choose $T_1\le t_2\le T_2$ so that $d(\gamma(t),\sigma(s_t))\le 1/2$ for every $t\ge t_2$ and $s_t\ge s_{t_2}$ for all $t\ge T_2$. By iterating this procedure we find a sequence of positive numbers
$$
0\le t_1\le T_1\le \dots\le t_n\le T_n\le \dots,
$$
so that for every $n\ge 0$ we have
\begin{alignat*}{2}
&d(\gamma(t),\sigma(s_t))\le 1/n,\qquad&&\forall t\ge t_n\\
&s_t\ge s_{t_n}, &&\forall t\ge T_n.
\end{alignat*}
Define $f(t):=s_t-t$. Then for $t\ge T_n$ we have
\begin{align*}
|f(t)-f(t_n)|&=|(s_t-s_{t_n})-(t-t_n)|\\
&=|d(\sigma(s_{t_n}),\sigma(s_t))-d(\gamma(t_n),\gamma(t))|\\
&\le d(\gamma(t_n),\sigma(s_{t_n}))+d(\gamma(t),\sigma(s_t))\\
&\le 2/n.
\end{align*}
In particular given $m> n$ we have $t_m\ge T_n$ and therefore $|f(t_m)-f(t_n)|\le 2/n$, which implies that  $c_n:=f(t_n)$ is a Cauchy sequence, and thus  it converges. Denote $c$ its limit.


We claim that for all fixed $n\geq 0$,
\begin{equation}
\label{truce}
\displaystyle \limsup_{t\to+\infty}d(\gamma(t),\sigma(t+ c_n))\le 3/n.
\end{equation}
Once this is proved,  the proposition follows since
$$d(\gamma(t),\sigma (t+ c))\leq d(\gamma(t),\sigma (t+ c_n))+|c_n-c|.$$

Given $t\ge T_n$ we have that 
\begin{align*}
d(\gamma(t),\sigma(t+ c_n))&=d(\gamma(t),\sigma(t+f(t_n)))\\
&\le d(\gamma(t),\sigma(s_t))+d(\sigma(s_t),\sigma(t+f(t_n)))\\
&\le 1/n+|f(t)-f(t_n)|\\
&\le 3/n,
\end{align*}
which implies \eqref{truce}.
\end{proof}

\section{Strongly pseudoconvex domains of \texorpdfstring{$\mathbb C^q$}{Cq}}\label{sec:strong}

\label{squeezing}
Given a bounded domain $D\subseteq \C^q$,   we denote its Kobayashi metric by  $d_D$, its Euclidean closure by $\overline{D}$, and its Euclidean boundary  by $\partial D$.


\begin{definition}[Squeezing function]
\label{DefSq}
Let $D\subseteq \C^q$ be a bounded domain. Given $z\in D$  define  the family of functions
\begin{alignat*}{2}
\mathcal{F}_z&:= \{\varphi : D\longrightarrow \mathbb{B}^q \; | \text{ $\varphi$ is holomorphic, injective and } \varphi(z)=0\}.
\end{alignat*}
The \textit{squeezing function} of $D$ at the point $z \in D$ is defined as  
$$
s_D(z):=\sup\{r>0 :  B(0,r) \subseteq \varphi(D) \text{ for some }\varphi\in\mathcal F_z\}.
$$
\end{definition}

By \cite[Theorem 2.1]{DGZpac} there always exist a map $\varphi\in\mathcal F_z$ realizing the supremum above.
In \cite[Theorem 1.3]{DGZ} F. Deng, Q. Guan, and L. Zhang proved that if $D$ is a bounded strongly pseudoconvex domain with boundary of class $C^2$ then $\displaystyle \lim_{z \to \partial D}s_D(z)= 1$.

\begin{lemma}\label{ladyhawke}
Let $D\subseteq \C^q$ be a bounded domain and let  $(z_n)$ be a sequence in $D$ so that $\lim_{n\to\infty}s_D(z_n)=1$.
Choose $\varphi_n\in\mathcal F_{z_n}$ so that $B(0,r_n)\subseteq D_n\subseteq \B^q$, where $D_n:=\varphi_n(D)$ and $r_n:=s_D(z_n)$.
Then 
 \begin{equation}\label{approx}
\displaystyle \lim_{n\to \infty} d_{D_n}(x,y)= d_{\B^q}(x,y),
\end{equation}
uniformly on compact subsets of $\B^q\times\B^q$.
\end{lemma}
\begin{proof}
The map $F_n(x)=r_nx$ is an isometry between the Euclidean balls $\B^q$ and $B(0,r_n)$, thus $
d_{B(0,r_n)}(x,y)=d_{\B^q}(r_n^{-1}x,r_n^{-1}y)$. 
 It is easy to see that
$$
\displaystyle \lim_{n\to \infty} d_{B(0,r_n)}(x,y)= d_{\B^q}(x,y),
$$
uniformly on compact subsets of $\B^q\times\B^q$. 
We obtain the result by noticing that $d_{B(0,r_n)}\ge d_{D_n}\ge d_{\B^q}$.
\end{proof}

A bounded domain $D$ satisfying $\displaystyle \lim_{z\to \partial  D} s_D= 1$ clearly  satisfies $\displaystyle \inf_{z\in D}s_D\left(z\right)>0$, that is, it is \textit{uniformly squeezing}.
By \cite[Theorem 1, Theorem 2]{Yeun} it follows that $D$ is Kobayashi complete, and thus proper and geodesic.

\begin{theorem}\label{raysgetstogether}
Let $D\subseteq \C^d$ be a bounded domain satisfying $\displaystyle \lim_{z\to \partial  D} s_D= 1$. 
Then $(D,d_D)$ has approaching geodesics.
\end{theorem}

We first need  some preparation.
Let $(\gamma_n:[-T_n,T_n]\longrightarrow D)$ be a sequence of geodesic segments, so that $T_n\to+\infty$.
Assume that there exist a sequence of points $(z_n)$ in $ D$ so that $ d_D(z_n,\gamma_n(0))\le C$ for some $C\geq 0$, and so that $\displaystyle \lim_{n\to \infty}s_D\left(z_n\right)=1$.
For every $n\ge0$ we choose $\varphi_n\in\mathcal F_{z_n}$ so that $B(0,r_n)\subseteq D_n\subseteq \B^q$, where $D_n:=\varphi_n(D)$ and $r_n:=s_D(z_n)$.
Denote $\widehat{\gamma}_n:=\varphi_n\circ \gamma_n$.
By Ascoli--Arzel\`a Theorem there exists a subsequence  $(\widehat{\gamma}_{n_k})$ converging uniformly on compact subsets to a geodesic line  $\eta:\mathbb R\rightarrow \B^q$.

\begin{lemma}\label{distance}
We have
$$
\inf_{t\in\mathbb R}d_{\B^q}(0,\eta(t))=\displaystyle \lim_{k\to\infty}\inf_{t\in[-T_{n_k},T_{n_k}]}d_{D_{n_k}}(0,\widehat\gamma_{n_k}(t)).
$$
\end{lemma}
\begin{proof} The sequence $(\widehat \gamma_{n_k})$ converges to $\eta$ uniformly on every compact subset $[-M,M]$.
Therefore the segments $\widehat\gamma_{n_k}([-M,M])$ and $\eta([-M,M])$ remain inside some compact subset of the ball.
Hence 
\begin{equation}\label{Ihatethis}
\displaystyle \lim_{k\to\infty}d_{D_{n_k}}(0,\widehat\gamma_{n_k}(t))=d_{\B^q}(0,\eta(t)),
\end{equation}
  uniformly on compact subsets of $\R$.

As a consequence, we obtain that
$$
\inf_{t\in\mathbb R}d_{\B^q}(0,\eta(t))\ge \displaystyle \limsup_{k\to\infty}\inf_{t\in[-T_{n_k},T_{n_k}]}d_{D_{n_k}}(0,\widehat\gamma_{n_k}(t)).
$$

On the other hand, given  $t_k\in[-T_{n_k},T_{n_k}]$ so that $$\inf_{t\in[-T_{n_k},T_{n_k}]}d_{D_{n_k}}(0,\widehat\gamma_{n_k}(t))=d_{D_{n_k}}(0,\widehat\gamma_{n_k}(t_k))$$ we have that 
$$
d_{D_{n_k}}(0,\widehat\gamma_{n_k}(t_k))\le d_{D_{n_k}}(0,\widehat\gamma_{n_k}(0))\le C,
$$
and thus by triangle inequality
$$
|t_k|=d_{D_{n_k}}(\widehat\gamma_{n_k}(t_k),\widehat\gamma_{n_k}(0))\le 2C.
$$
Since $t_k$ remains in a compact subset of the real line, by \eqref{Ihatethis} we conclude that
\begin{align*}
\inf_{t\in\mathbb R}d_{\B^q}(0,\eta(t))&\le \displaystyle \liminf_{k\to\infty}d_{\B^q}(0,\eta(t_k))\\
&\le \displaystyle \liminf_{k\to\infty}d_{D_{n_k}}(0,\widehat\gamma_{n_k}(t_k))\\
&= \displaystyle \liminf_{k\to\infty}\inf_{t\in[-T_{n_k},T_{n_k}]}d_{D_{n_k}}(0,\widehat\gamma_{n_k}(t)),
\end{align*}
which proves \eqref{distance}.
\end{proof}

\begin{proof}[Proof of Theorem \ref{raysgetstogether}]
Suppose by contradiction that the theorem is false.  Then there exist asymptotic geodesic rays    $\gamma_1$ and $\gamma_2$  which are not strongly asymptotic. Since they are asymptotic, there exists $C>0$ such that
$$d_D(\gamma_1(t),\gamma_2(t))\leq C$$
for all $t\geq0$.
Moreover, by Proposition \ref{prop:approach}, there  exist a positive and divergent sequence $(t_n)$ such that  for all $n\geq 0$,
$$
c\le \inf_{s\geq 0} d_D\left(\gamma_1(t_n),\gamma_{2}(s)\right)\le d_D(\gamma_1(t_n),\gamma_2(t_n))\le C,
$$
for some $0<c\le C$. 
Denote $z_n:=\gamma_1(t_n)$.

Consider the sequences of re-parametrized geodesic rays $(\gamma^{1}_n:[-t_n,+\infty)\longrightarrow D)$ and $(\gamma^{2}_n:[-t_n,+\infty)\longrightarrow D)$ given by 
$$\gamma^{1}_n(t):=\gamma_1(t+t_n),\qquad \gamma^{2}_n(t):=\gamma_2(t+t_n).
$$ 
It follows that 
$$
\gamma^1_n(0)=z_n,\qquad d_D(\gamma^2_n(0),z_n)\le C.
$$

For every $n\ge0$ we choose $\varphi_n\in\mathcal F_{z_n}$ so that $B(0,r_n)\subseteq D_n\subseteq \B^q$, where $D_n=\varphi_n(D)$ and $r_n=s_D(z_n)$.  
Let $\widehat\gamma^i_n:=\varphi_n\circ\gamma^i_n$.  
After taking a subsequence if necessary we may assume that for $i=1,2$ the sequence  $(\widehat\gamma^{i}_n)$ converges uniformly on compact subsets to a geodesics line  $\eta^{i}$ of $\B^q$. 
Furthermore by Lemma \ref{distance} we have $\eta^1(0)=0$ and
$$
\inf_{t\in\mathbb R}d_{\B^q}(0,\eta^2(t))=\displaystyle \lim_{n\to \infty}\inf_{s\ge -t_n}d_{D_n}(0,\widehat\gamma^{2}_n(s))= \displaystyle \lim_{n\to \infty}\inf_{s\ge 0}d_{D}(\gamma_1(t_n),\gamma_2(s))\geq c,$$
which implies that the images of the two geodesic lines are distinct.

We claim that $ \sup_{\R}d_{\B^q}(\eta^1(t),\eta^2(t))\le C.$
Indeed given a positive  number $M>0$, we have $t_n\ge M$ for every $n$ sufficiently large. Therefore $d_{D_n}(\widehat\gamma^1_n(t),\widehat\gamma^2_n(t))\leq C$ for every $t\in [-M,M]$.
By uniform convergence, the curves $\widehat\gamma^i_n:[-M,M]\rightarrow \B^q$ remain in a compact subset of the unit ball, and therefore by \eqref{approx} we obtain that 
$$
\sup_{-M\le t\le M}d_{\B^q}(\eta^1(t),\eta^2(t))\le C.
$$
Hence  there exist $\zeta\neq \xi\in \partial \B^q$ such that for $i=1,2$ we have
$$\displaystyle \lim_{t\to +\infty}\eta^{i}=\zeta,\quad \displaystyle \lim_{t\to -\infty}\eta^{i}=\xi,$$contradicting the fact that in $\B^q$ there exists a unique geodesic line (up to reparametrization) connecting two different boundary points.
\end{proof}

\begin{corollary}\label{ohyeah1.5.1}
Let $D\subseteq \C^q$ be a bounded domain. Assume that $D$ is either
\begin{enumerate}
\item strongly pseudoconvex with $C^2$ boundary, or
\item  convex with $\displaystyle \lim_{z\to \partial D} s_D(z)=1$.
\end{enumerate}
Then the horofunction compactification $ \overline{D}^H$, the Gromov compactification $\overline{D}^G$  and the Euclidean compactification $\overline D$ are topologically equivalent.
\end{corollary}

\begin{proof}
We claim that $\left(D,d\right)$ is Gromov hyperbolic and that $\overline{D}^G$ and $\overline D$ are topologically equivalent. 
Indeed, in case $(1)$ this is proved in   \cite{BaBo} and in case $(2)$ it follows from \cite[Lemma 20.8]{Zim6} combined with \cite[Theorem 1.5]{BrGaZi}. 
In both cases $\displaystyle \lim_{z\to \partial D} s_D(z)=1$ (for $(1)$ this is  \cite[Theorem 1.3]{DGZ}). 
Thus by Theorem \ref{raysgetstogether} the metric space $\left(D,d\right)$ has approaching geodesics. 
By Theorem \ref{abstract}, we conclude that $\overline{D}^G$ and $ \overline{D}^H$ are topologically equivalent. 
\end{proof}
\begin{remark}
In the class of convex domains there is a  gap between strong pseudoconvexity and the condition $\displaystyle \lim_{z\to \partial D} s_D(z)=1$. Indeed, when $D$ is a bounded convex domain with boundary of class $C^{2,\alpha}$ ($\alpha >0$) and satisfies $\displaystyle \lim_{z\to \partial D} s_D(z)=1$, then $D$ is strongly pseudoconvex, see \cite[Theorem 1.5]{Zim4}. However there exists a bounded convex domain which is not strongly pseudoconvex, has boundary of class $C^2$ and satisfies $\displaystyle \lim_{z\to \partial D} s_D(z)=1$, see \cite{FoWo}. 
\end{remark}

\section{Convex domains of finite type of \texorpdfstring{$\C^{q}$}{Cq}}\label{sec:finite}
Given a convex domain $D\subseteq \C^q$,  we denote by $d_D$ its Kobayashi (pseudo)metric. First we recall the definition of  finite line type for convex domains.

\begin{definition}
\label{Def_Type}
Let $D\subseteq \C^q$ be a convex domain. Let $L\geq 2$. We say that $p\in\partial D$ has \textit{finite line type} $L$  if  
\\$\bullet$ $\partial D$ is of class $C^L$ in a neighborhood of $p$, and
\\$\bullet$ $L$ is the maximum of the set
$$S_p:=\left\{\nu(\rho\circ f):f:\C\longrightarrow\C^{q}\ \ \mbox{non trivial affine,}\ \ f(0)=p\right\},$$
where $\rho$ is a defining function of class $C^L$ for $\partial D$ in a neighborhood of $p$ and $\nu(\rho \circ f)$ denotes the order of vanishing of $\rho \circ f$ at $0$.

We say that $D$ has \textit{finite line  type} $L$ if  the line type  of all $x\in \partial D$ is at most $L$ and this bound is realized at some boundary point. 
\end{definition}

\begin{remark}
When  $D\subseteq \C^q$ is a convex domain with $C^\infty$ boundary,  McNeal \cite{mcneal} proved that a point $x\in \partial D$ has finite line type $L$ if and only if it has finite D'Angelo type $L$.
\end{remark}

In this section we prove the following result:
\begin{theorem}\label{mainfinitetype}
Let  $D\subseteq \C^{q}$ be a bounded convex domain of finite line type.
Then Gromov compactification $\overline{D}^G$, the horofunction compactification $\overline{D}^H$, and the Euclidean compactification  $\overline{D}$ are topologically equivalent.
\end{theorem}
Zimmer proved in \cite{Zim1} that $(D,d_D)$  is Gromov hyperbolic, and that the Gromov compactification $\overline{D}^G$ is topologically equivalent to the Euclidean compactification $\overline{D}$. 
Thus Theorem \ref{mainfinitetype} holds once we prove that $\overline{D}^G$ is topologically equivalent to $\overline{D}^H$.
By Theorem \ref{weaknotion},  it  is enough to prove the following result.
\begin{proposition}\label{asyfintyp}
Let  $D\subseteq \C^{q}$ be a bounded convex domain of finite line type.  Let $\mathscr{F}\subset\mathscr{R}(X)$ be the subfamily of geodesic rays which are contained in a complex geodesic. Then $(D,k_D)$ satisfies the weak approaching geodesics property relative to $\mathscr{F}.$  
\end{proposition}


The proof of Proposition \ref{asyfintyp} requires some preliminary results. 
\begin{theorem}[{\cite[Chapter III.H, Theorem 3.9]{BH}}]\label{extension}
	Let $(X,d_X)$ and $(Y,d_Y)$ be two proper  geodesic Gromov hyperbolic metric spaces and  let $i:X\longrightarrow Y$ be an isometric embedding. Then $i$ extends to a topological embedding (also denoted by $i$)  $\overline{X}^G\to\overline{Y}^G$.
\end{theorem}

Let $X,Y$ be two topological spaces, and let $C(X,Y)$ denote the set of continuous functions between $X$ and $Y$, equipped with the compact-open topology. Recall that when $Y$ is a metric space, a sequence converges in the compact-open topology if and only if it converges uniformly on every compact subsets of $X$.

\begin{proposition}\label{isoGr}
Let $(X,d_X)$ and $(Y,d_Y)$ be two proper geodesic Gromov hyperbolic metric spaces and  let $(i_n\colon X\to Y)$ be a sequence of isometric embeddings, converging in $C(X,Y)$ to  the isometric embedding $i$. Then $(i_n)$ converges to $i$ in $C(\overline{X}^G,\overline{Y}^G)$.
\end{proposition}
\begin{proof}
 Since the Gromov compactification of a hyperbolic space is metrizable we can choose 
 metrics  $d_{\overline X}$ and $d_{\overline Y}$   inducing the topology on ${\overline X}^G$ and ${\overline Y}^G$ respectively.
Assume by contradiction that the conclusion of the proposition does not hold, that is $(i_n)$ does not converge $d_{\overline Y}$-uniformly on ${\overline X}^G$ to $i$.
Then there exists $\varepsilon>0$, a sequence $(z_n)$ in ${\overline X}^G$ and a subsequence of $(i_n)$, still denoted $(i_n)$ for simplicity, such that $$d_{\overline Y}(i_n(z_n),i(z_n))>\varepsilon.$$
Up to passing to subsequences we may assume that $z_n$ converges to a point $x\in {\overline X}^G$, and since $i$ is continuous, we have that $i(z_n)$ converges to $i(x)$. Moreover we can assume that 
\begin{equation}\label{terramare} \displaystyle \lim_{n\to\infty}i_n(z_n)= y\neq i(x) .
\end{equation}
Assume first that $x\in X$. Since $i_n$ converges to $i$ $d_Y$-uniformly on compact subsets of $X$, we have that  
$i_n(z_n)\to i(x)$, contradicting \eqref{terramare}.
Therefore we must have $x\in\partial_G X$.

 For all $n\geq 0$, let $x_n$ be a point in $X$ such that 
$$d_{\overline X}(z_n,x_n)\leq \frac{1}{n}\qquad\text{and}\qquad d_{\overline{Y}}(i_n(z_n), i_n(x_n))\leq \frac{1}{n}.$$ Clearly
  $$\displaystyle \lim_{n\to\infty}x_n= x\qquad\text{and}\qquad \displaystyle \lim_{n\to\infty}i_n(x_n)= y.$$  
Fix a point $p\in X$. We have that $d_X(p,x_n)\to +\infty$. As a consequence
$d_Y(i(p),i(x_n))\to+\infty$ and $d_Y(i_n(p),i_n(x_n))\to+\infty$. Since $i_n(p)$ is uniformly bounded, we conclude that  $i(x),y\in\partial_GY$.

Since $(i_n)$ converges to $i$ in $C(X,Y)$, we have that $i_n(x_k)\to i(x_k)$ for all $k\ge 0$. Let $n_k$ be such that $d_{\overline Y}(i_{n_k}(x_k),i(x_k))<\frac{1}{k}$.
Then the sequence $(i_{n_k}(x_k))$ converges to $i(x)$.

Furthermore we may find $R>0$ so that $d_Y(i_n(p),i(p))<R$ for all $n\geq 0$.
We conclude that
$$
(i_{n_k}(x_{n_k})\vert i_{n_k}(x_k))_{i(p)}\ge (i_{n_k}(x_{n_k})\vert i_{n_k}(x_k))_{i_{n_k}(p)}-R= (x_{n_k}\vert x_k)_{p}-R.
$$
Since the term on the right converges to infinity, by Remark \ref{infinitysequences} we conclude that the sequences $(i_{n_k}(x_{n_k}))$ and $(i_{n_k}(x_k))$  converge to the same point in the Gromov boundary, and thus that $i(x)=y$, giving a contradiction.
\end{proof}
\begin{definition}\label{quasig}
Fix $A\geq1$, $B\geq0$. 
A (non-necessarily continuous) map $\sigma\colon \R_{\geq 0}\to X$ is  a $(A,B)$-\textit{quasi-geodesic ray} if for every $t,s\geq0$
$$A^{-1}|t-s|-B\leq d(\sigma(t),\sigma(s))\leq A|t-s|+B.$$
Similarly one can define quasi-geodesics lines $\sigma\colon \R\to X$. 
\end{definition}
\begin{lemma}\label{geoqgeo}
	Let $(X,d)$ be a proper geodesic Gromov hyperbolic space.
	Let $\gamma:\R_{\geq0}\longrightarrow X$ be a geodesic ray and $\sigma:\R_{\geq0}\longrightarrow X$ be a $(1,B)$-quasi-geodesic ray with $\lim_{t\to+\infty}\sigma(t)=[\gamma]\in\partial_GX$.
	Then there exists $M>0$ such that 
	$$\sup_{t\geq0}d(\gamma(t),\sigma(t))<M.$$
\end{lemma}
\proof
By the Gromov shadowing Lemma \cite[Chapter 5]{GdlH}, there exists $R\ge 0$ so that the quasi-geodesic ray $\sigma$ lies in a $R$-neighborhood of $\gamma$. 
More precisely, for every $t\geq0$ there exists $s_t\geq0$ such that
$$d(\gamma(s_t),\sigma(t))=\inf_{s\geq0}d(\gamma(s),\sigma(t))<R.$$
By \cite[Proposition 7.6]{AAG} we obtain
$$d(\gamma(0),\sigma(t))\geq d(\gamma(0),\gamma(s_t))+d(\gamma(s_t),\sigma(t))-6\delta\geq s_t-6\delta. $$
Consequently, if we write $D:=d(\sigma(0),\gamma(0))$
\begin{align*}t-s_t&\geq d(\sigma(0),\sigma(t))-B-d(\gamma(0),\sigma(t))-6\delta\geq-B -d(\sigma(0),\gamma(0))-6\delta=-D-B-6\delta\\[5pt]
t-s_t&\leq d(\sigma(0),\sigma(t))+B-d(\gamma(0),\gamma(s_t))\leq d(\sigma(0),\gamma(0))+d(\sigma(t),\gamma(s_t))+B\leq D+R+B.
\end{align*}
The two inequalities above imply that
$$|t-s_t|\leq D+B+\max\{6\delta,R\},$$
and therefore
$$d(\gamma(t),\sigma(t))\leq d(\gamma(t),\gamma(s_t))+d(\gamma(s_t),\sigma(t))\leq |t-s_t|+R\leq R+D+B+\max\{6\delta,R\}.$$
\endproof
The following theorem describes a scaling process for domains of finite line type.
We refer to the papers of Gaussier and Zimmer \cite{Gau,Zim1} for a more detailed description of scaling techniques.
Recall that a  polynomial $P\colon \C^q\rightarrow\R$ is \textit{non-degenerate} if the set $\{P=0\}$ contains no complex lines.


\begin{theorem}\label{scaling}
Let  $D\subseteq \C^{q}$ be a convex domain  such that $\partial D$ has finite line type near some $\xi\in\partial D$.
Let $\left(x_n\right)$ be a sequence in $D$ converging to $\xi$. Then there exists a divergent subsequence $(n_k)$ and affine maps $A_k\in{\rm Aff}(\C^{q})$ such that
\begin{enumerate}
	\item $\left(A_kD\right)$ converges local Hausdorff to the $\C$-proper convex domain 
	$$D_P:=\{(z,w)\in\C^{q-1}\times\C: \Re w>P(z)\} ,$$
	here $P:\C^q\longrightarrow\R$ is a non-negative non-degenerate convex polynomial with $P(0)=0$,
	\item $(A_k x_{n_k})$ converges to $x_\infty\in D_P$,
	\item $\left(d_{A_kD}\right)$ converges to $d_{D_P}$ uniformly on compact subsets of $D_P\times D_P$.
\end{enumerate}
\end{theorem}

\begin{lemma}
	Let $P:\C^{q-1}\longrightarrow\R$ be a non-negative non-degenerate convex polynomial with $P(0)=0$, then $W:=\{z\in\C^{q-1}:P(z)=0\}$ is a totally real linear subspace of $\C^{q-1}$.
\end{lemma}
\begin{proof}
$\newline$
$\bullet$ Clearly $0\in W$.
\\$\bullet$ If $z\in W$ and $t\in \mathbb{R}$, then $tz\in W$. 
Indeed the function $t\in\R\mapsto p_z(t):=P(tz)$ is a non-negative real convex polynomial with $p_z(0)=p_z(1)=0$. 
Therefore $p_z\equiv0$, and $tz\in W$.
\\ $\bullet$ The sum of two element of $W$ is also in $W$. Indeed for $z,w\in W$ we have
$$0\leq P(z+w)=P\left(\frac{1}{2}2z+\frac{1}{2}2w\right)\leq\frac{1}{2}P(2z)+\frac{1}{2}P(2w)=0,$$
proving that $P(z+w)=0$. 
This proves that $W$ is a real linear subspace of $\mathbb C^{q-1}$. 
\\$\bullet$ For every $z\in W\smallsetminus\{0\}$, we have $P(iz)\neq 0$. 
Indeed, if it were not the case, then $P$ would vanish on the complex line $\mathbb Cz$, contradicting its non-degeneracy. 
Hence $W$ is totally real.
\end{proof}

We want to study the limit of complex geodesics after a scaling in the normal direction. 
Let $\xi\in\partial D$. After an affine change of variable, we may assume \cite{Gau} that $\xi=0$ and that $\partial D$ has the following defining function in a neighborhood of the origin
\begin{equation}\label{normalform}
r(z,w)=-\Re w+H(z)+R(z,w).
\end{equation}
Here $H:\C^{q-1}\longrightarrow\R$ is a convex non-degenerate non-negative weighted homogeneous polynomial. This means that there exists $(\delta_1,\dots,\delta_{q-1})\in\N^{q-1}$ with $\delta_j\geq2$ such that for $t\geq0$ and $z\in\C^{q-1}$
$$H\left(t^{1\slash\delta_1}z_1,\dots,t^{1\slash\delta_{q-1}}z_{q-1}\right)=tH(z).$$
The remainder satisfies
$$R(z,w)=o\bigg(|w|+\sum_{j=1}^{q-1}|z_j|^{\delta_j}\bigg).$$

By \cite[Proposition 11.1]{Zim1}, there exists $\varepsilon>0$ such that the normal segment 
\begin{equation}\label{defsigma}
\sigma:t\in\R_{\geq0}\longmapsto (0,\varepsilon e^{-t})\in D
\end{equation}
is a $(1,B)$-quasi-geodesic ray of $(D,d_D)$. Let $t_n\nearrow+\infty$ and write $\lambda_n:=e^{t_n}$. 

The sequence $(\sigma(t_n))$ converges to the origin normally, and in this case the affine maps of Theorem \ref{scaling} have an explicit formula (see \cite{Gau}).
Indeed, up to taking a subsequence of $(t_n)$ if necessary, we may choose invertible matrices $B_n\in$GL$(\C^{q-1})$, so that the affine maps $A_n\in{\rm Aff}(\C^{q})$ given by
\begin{equation}
\label{affine}
A_nx=\begin{bmatrix} B_n & 0\\ 0 & \lambda_n
\end{bmatrix}x,
\end{equation}
satisfy the assumptions of Theorem \ref{scaling} with respect to the sequence $(\sigma(t_n))$,
 and
$$(A_n \sigma(t_n))\to (0,\varepsilon).$$
By \cite[Proposition 2.2]{Gau} the limit domain given by part (1) of Theorem \ref{scaling} is equal in this case to 
$$D_H:=\{ \Re w>H(z)\}.$$ 

\begin{remark}\label{compactificationseq}
	The domain $D_H$ equipped with its Kobayashi metric is Gromov hyperbolic \cite[Theorem 1.6]{Zim5}.
	The Gromov compactification of a Gromov hyperbolic, $\mathbb C$-proper convex domain $D$ is equivalent to the \emph{Euclidean end compactification} $\overline D^*$, see \cite[Theorem 1.5]{BrGaZi}. 
	One can prove that the compactification $\overline{D_H}^*$ is equal to the one-point compactification $\overline{D_H}\cup\{\infty\}$, where $\overline{D_H}$ denotes the Euclidean closure. 
\end{remark}

In the remaining of the section, we always consider complex geodesics defined on the right half-plane $\H:=\{\zeta\in\C:\Re \zeta>0\}$, and not on the unit disk as customary. 
Given a complex geodesic $\phi:\H\rightarrow D$, we write
$$
\phi(\zeta)=(\phi_1(\zeta),\phi_2(\zeta))\in\mathbb C^{q-1}\times\mathbb C.
$$

\begin{lemma}
\label{lemma:radiallimit}
Let $\phi:\H\rightarrow D$ be a complex geodesic satisfying $\phi(0)=0$.
Then the non-tangential limit of $\phi_2'$ at $0$ exists and is a positive real number.
\end{lemma}
\begin{proof}
Let $\pi:\C^q\rightarrow \C$ denote the projection to the last coordinate. 
By convexity of the domain $D$, we have $\pi(D)\subset \H$.

The curve $\phi(e^{-t})$ is a geodesic ray with endpoint the origin.
Recall that the curve $\sigma$ defined in \eqref{defsigma} is a $(1,B)$-quasi-geodesic, therefore by Lemma \ref{geoqgeo} we may find $M>0$ so that 
\begin{equation}\label{eq1}
\sup_{t\ge0}d_D(\sigma(t),\phi(e^{-t}))<M.
\end{equation}

By non-expansiveness of the Kobayashi metric, it follows that
$$
\sup_{t\ge 0}d_{\mathbb H}(\pi\circ\sigma(t),\pi\circ\phi(e^{-t}))=\sup_{t\ge 0}d_{\mathbb H}(\varepsilon e^{-t},\phi_2(e^{-t}))<M.
$$
This shows that for every $0<r<1$ we have
$$
d_{\mathbb H}(r,\phi_2(r))\le d_{\mathbb H}(r, \varepsilon r)+d_{\mathbb H}(\varepsilon r,\phi_2(r))\le |\log\varepsilon|+M.
$$

As a consequence, the dilation of the map $\phi_2$ at the point $0\in\partial \H$ is bounded from above. Indeed we have  
$$
\liminf_{z\to 0}d_{\mathbb H}(1,z)-d_{\mathbb H}(1,\phi_2(z))\le \liminf_{r\to 0}d_{\mathbb H}(r,\phi_2(r))<|\log\varepsilon|+M.
$$
On the other hand, it is simple to show that the left hand side is bounded from below by $-d_{\H}(1,\phi_2(1))$.
We conclude that there exists $c>0$ so that
$$
\liminf_{z\to 0}d_{\mathbb H}(1,z)-d_{\mathbb H}(1,\phi_2(z))=\log c.
$$
By the classical Julia-Wolff-Carath\'eodory Theorem (see \cite[Theorem 1.2.7]{AbBook}), the non-tangential limit of $\phi_2'$ at $0$ exists and is equal to $c$.
Recall here that the chosen normalization of the Poincar\'e metric differs from the one adopted by Abate and other authors by a factor $1/2$.
\end{proof}
For the sake of simplicity we will write $\phi_2'(0)$ for the non-tangential limit of $\phi_2'$ at $0$. Up to pre-composing the complex geodesic $\phi$ with a dilation of $\mathbb H$, we may always assume that $\phi_2'(0)=1$.

\begin{lemma}\label{scalingcg}
Let $\phi:\H\longrightarrow D$ be a complex geodesic satisfying $\phi(0)=0$ and $\phi_2'(0)=1$.
Then $(A_n\phi(\lambda_n^{-1}\zeta))$ converges uniformly on compact subsets to the complex geodesic $\widehat{\phi}:\H\longrightarrow D_H$ given by
\begin{equation}\label{limitgeo}
\widehat{\phi}(\zeta)=\left(0,\zeta\right).
\end{equation}
\end{lemma}
\proof
 As shown in the proof of the previous lemma we may find $M>0$ so that 
$$
\sup_{t\ge0}d_D(\sigma(t),\phi(e^{-t}))<M.
$$

Since $$d_{A_nD}(A_n\varphi(\lambda_n^{-1}),A_n\sigma(t_n))<M,$$
 by points (2) and (3) of Theorem \ref{scaling} we have that  the sequence $(A_n\phi(\lambda_n^{-1}))$ is relatively compact in $D_H$.
Write $$\phi_n(\zeta):=A_n\phi(\lambda_n^{-1}\zeta).$$
By Ascoli-Arzel\'a Theorem, every subsequence of $(\phi_n)$ admits a subsequence converging uniformly on compact subsets. 
The proposition follows once we prove that \eqref{limitgeo} is the only possible limit of a subsequence of $(\phi_n)$.

Consider a convergent subsequence, which we still denote by $(\phi_n)$ for simplicity. 
Using the fact that $d_{A_nD}\rightarrow d_{D_H}$ uniformly on compact subsets, we conclude that the limit of the sequence $(\phi_n)$ is a complex geodesic $\widehat{\phi}\colon\H\longrightarrow D_H$.

By Lemma \ref{extension} and Remark \ref{compactificationseq}, the complex geodesic $\widehat{\phi}$ extends to a  topological embedding from $\overline{\mathbb H}\cup \{\infty\}$ to $\overline{D_H}\cup \{\infty\}$, still denoted as $\widehat\phi$. 
Notice that  by \eqref{affine} we have that $A_n\sigma(t_n+t)=(0,e^{-t})$. 
By \eqref{eq1}, we obtain that 
$$
d_{D_H}((0,e^{-t}),\widehat\phi(e^{-t}))=\lim_{n\to\infty} d_{D_n}(A_n\sigma(t+t_n),A_n\phi(e^{-t-t_n}))\le M\qquad\forall t\in\mathbb R,
$$ proving that the geodesic lines $(0,e^{-t})$ and $\widehat\phi(e^{-t})$ are asymptotic  both for $t\geq 0$ and for $t\leq 0$. 
This implies that $\widehat\phi(0)=0$ and $\widehat\phi(\infty)=\infty$.
In particular the map $\widehat\phi$ is a topological embedding from the Euclidean closure $\overline{\mathbb H}$ to the Euclidean closure $\overline{D_H}$.

Since the non-tangential limit $\phi_2'(0)=1$, for every $\zeta\in\H$ we have
$$\frac{d}{d\zeta}(\phi_n)_2(\zeta)=\frac{d}{d\zeta}\lambda_n\phi_2(\lambda_n^{-1}\zeta)=\phi_2'(\lambda_n^{-1}\zeta)\to 1.$$
Consequently $(\widehat{\phi}_2)'(\zeta)=1$ for all $\zeta\in \H$, and therefore $\widehat\phi_2={\rm id}$.
We conclude that the complex geodesic $\widehat{\phi}\colon\H\rightarrow D_H$ has the form
$$\widehat{\phi}(\zeta)=(\widehat{\phi}_1(\zeta),\zeta).$$


We want to prove that $\widehat{\phi}_1\equiv0$. After a $\C$-linear change of coordinates, we may assume that the zero set of $H$ is contained in $\R^{q-1}$. Note that $H(\widehat{\phi}_1(\zeta))\leq\Re \zeta $ for every $\zeta\in\overline{\H}$, therefore $\widehat{\phi}(\partial\H)\subseteq\{H=0\}\times\R\subseteq\R^{q}$. By the Schwarz reflection principle, the map $\widehat{\phi}$ extends to a holomorphic map on $\mathbb C$, defined by the functional equation $\widehat\phi(-x+iy)=\overline{\widehat\phi(x+iy)}$.

Given $\mu_n\nearrow+\infty$,  consider the automorphisms  of $D_H$ defined by
$$\Psi_n(z,w):=(\Lambda_n z,\mu_n^{-1}w),$$ where $\Lambda_n:={\rm diag}\left(\mu_n^{-\frac{1}{\delta_1}},\dots,\mu_n^{-\frac{1}{\delta_{q-1}}}\right)$. 
Consider the sequence of holomorphic maps $ \widehat\psi_n\colon \C\to \C^q$ defined by
$$\widehat\psi_n(\zeta):=\Psi_n(\widehat{\phi}(\mu_n\zeta))=(\Lambda_n\widehat{\phi}_1(\mu_n\zeta),\zeta) .$$
When restricted to $\H$, every  map $\widehat\psi_n$ is a  complex geodesics of $D_H$. Thus,
up to extracting a subsequence,  the sequence $(\widehat\psi_n|_{\H})$
converges uniformly on compact subsets of $\H$ to a complex geodesic of $D_H$. Here we use again the fact that the geodesic lines $(0,e^{-t})$ and $\widehat\phi(e^{-t})$ are asymptotic for $t\leq 0$, and therefore that $(\widehat\psi_n(1))$ remains at finite $d_{D_H}$-distance  from the point $(0,1)$. By Proposition \ref{isoGr}, the sequence $(\widehat\psi_n)$ converges uniformly on compact subsets of $\overline\H$.  By the Schwarz reflection principle  the sequence $(\widehat\psi_n)$ converges uniformly on compact subsets of the closed left half-plane $-\overline{\H}$, and thus it  converges uniformly on compact subsets of all $\C$.

Writing the Taylor expansion of $\widehat{\phi}_1$ at $0$ 
$$\widehat{\phi}_1(\zeta)=\sum_{k\geq1}a_k\zeta^k, \ \zeta\in\C ,$$ with $a_k\in\C^{q-1}$, we obtain, for all $n\geq 0$,
$$\Lambda_{n}\widehat{\phi}_1(\mu_n\zeta)=\sum_{k\geq1}\mu_n^k\Lambda_na_k\zeta^k.$$
Since the sequence $(\mu_n^k\Lambda_{n})_{n\geq 0}$ diverges for every fixed $k$,  the sequence of holomorphic maps $(\Lambda_{n}\widehat{\phi}_1(\mu_n\zeta))$ converges uniformly on compacts subsets of  $\C$ to some holomorphic map if and only if $a_k=0$ for every $k\geq 1$ i.e. $\widehat{\phi}_1\equiv0$.
\endproof
We are ready to prove Proposition \ref{asyfintyp}

\begin{proof}[Proof of Proposition \ref{asyfintyp}]
The family $\mathscr{F}$ satisfies point $(1)$ of Definition \ref{def_weaknotion} by Proposition \ref{isoGr}. Thus it is enough to prove point $(2)$ of Definition \ref{def_weaknotion}.
Recall that $\left(D,k_D\right)$ is Gromov hyperbolic and $\overline{D}$ and $\overline{D}^G$ are topologically equivalent.
We need to prove that, if $\phi,\psi:\H\rightarrow D$ are two complex geodesics with $\phi(0)=\psi(0)=\xi\in \partial D$, then the two geodesic rays $\phi(e^{-t})$ and $\psi(e^{-t})$ are strongly asymptotic.
After a change of coordinates, we may assume that $\xi=0$, and that the defining function of $\partial D$ near $0$ is as in \eqref{normalform}.

By Lemma \ref{lemma:radiallimit} we have $\phi_2'(0)=c_1>0$ and $\psi'_2(0)=c_2>0$.
Without loss of generality, we may replace the two complex geodesics with  $\psi(\frac{\zeta}{c_1})$ and $\phi(\frac{\zeta}{c_2})$.
Therefore we may assume that $\phi_2'(0)=\psi_2'(0)=1$.

By contradiction, suppose that there exists $c>0$ and $t_n\nearrow +\infty$ such that
\begin{equation}\label{fintypcontr}
\inf_{s\in[0,1)}d_D(\phi(e^{-t_n}),\psi(e^{-s}))\geq c.
\end{equation}
Write $\lambda_n:=e^{t_n}$ and let $(A_n)$ be a sequence in ${\rm Aff}(\mathbb C^{q})$ as in \eqref{affine}.
By Lemma \ref{scalingcg} both sequences $(A_{n}\phi(\lambda_{n}^{-1}\zeta))$ and $(A_{n}\psi(\lambda_{n}^{-1}\zeta))$  converge uniformly on compact subsets to the complex geodesic $\widehat\phi(\zeta)=(0,\zeta)$ of the domain $D_H$.

Let $D_n=A_n D$. Then by \eqref{fintypcontr} we obtain that
$$
d_{D_{n}}(A_{n}\phi(\lambda_{n}^{-1})\,,\,A_{n}\psi(\lambda_{n}^{-1}))=d_D(\phi(e^{-t_{n}})\,,\,\psi(e^{-t_{n}}))\ge c.
$$
On the other hand both sequences $(A_{n}\phi(\lambda_{n}^{-1}))$ and $(A_{n}\psi(\lambda_{n}^{-1}))$ converge to the point $(0,1)$. By point (3) of Theorem \ref{scaling} we conclude that the left term in the inequality above converge to $0$, giving a contradiction.
\end{proof}

\section{Dynamics of non-expanding self-maps}

The Denjoy--Wolff theorem is a classical result describing the dynamics of holomorphic self-maps of the disc $\D\subseteq \C$. 
It has been generalized to the ball $\B^q\subseteq \C^q$ by Herv\'e \cite{He} and to bounded strongly convex domains with $C^2$ boundary by Abate \cite{Ab1988}. 
More recently, Karlsson \cite{Kar} proved the following version of the Denjoy--Wolff theorem for a non-expanding self-map of a Gromov hyperbolic metric space. Recall that a self-map $f\colon X\to X$ is \textit{non-expanding} iff $d(f(x),f(y))\leq d(x,y)$ for all $x,y\in X$.
\begin{theorem}[Karlsson]\label{DW}
Suppose $(X,d)$ is a proper  Gromov hyperbolic metric space. Let $f:X\to X$ be a non-expanding self-map with escaping  forward orbits. Then there exists a unique $\zeta\in\partial_GX$, called the \textit{Denjoy--Wolff point} of $f$, so that for all $x\in X$
	$$\lim_{n\to\infty}f^n(x)=\zeta.$$
\end{theorem}
\begin{remark}
\label{Rem_dicho}
In a proper metric space, a non-expanding self-map satisfies the following dynamical dichotomy: either all forward orbits are bounded, or all forwards orbits  \textit{escape}, that is they do not admit  bounded subsequences \cite{Calka}.
\end{remark}

Motivated by Karlsson's result, we show in this section  that relevant parts of the dynamical theory of holomorphic self-maps of the ball $\B^q$ (or more generally of bounded strongly convex domains of $\C^q$) extend naturally  to the case of non-expanding self-maps of   a proper geodesic Gromov hyperbolic metric space $X$. 
 We generalize classical concepts like  $K$-limits,  dilations, and boundary regular fixed points, and we  introduce to this general setting the study of the dynamics of backward orbits. Backward dynamics of a holomorphic self-map was  studied in the unit disc  by Bracci \cite{Br} and Poggi-Corradini \cite{PoCo1,PoCo2},  in the unit ball by Ostapyuk \cite{O}, and in bounded strongly convex domains by Abate--Raissy \cite{AbRaback,errata} (see also \cite{AAG, Ar2, AG} for the theory of canonical models associated with backward orbits).
 If  the compactifications $\overline X^H$ and $\overline X^G$ are topologically equivalent, we obtain stronger results, and in particular a generalization of the classical Julia's Lemma.

Let $(X,d)$ be a proper geodesic Gromov hyperbolic metric space.
Recall from  \eqref{WW} that the Gromov and horofunction compactifications satisfy $\overline X^G\ge \overline X^H$.
This means that there exists a continuous map $\Phi:\overline X^H\rightarrow \overline X^G$ such that ${\rm id}_X=\Phi\circ i_H$, where $i_{H}\colon X\rightarrow \overline X^H$ denotes the embedding of the space $X$ into the horofunction compactification.
For the sake of simplicity in this section we identify the point $i_H(x)$ with $x$ itself, and write $\overline X^{H}=X\sqcup \partial_{H}X$.
As a consequence, the restriction $\Phi|_X$ is the identity.

We will use the letters $H$ and $G$ to distinguish between the topologies of the two (possibly different) compactifications. 
In particular we say that a sequence is $H$-convergent (respectively $G$-convergent) if it is convergent with respect to the topology of $\overline X^H$ (respectively $\overline X^G$).
Notice that if $(a_n)$ is a sequence in $\overline X^H$ that $H$-converges to a point $a\in\overline X^H$, then the sequence $(\Phi(a_n))$  $G$-converges to $\Phi(a)$.

Furthermore, given a subset $S\subseteq X$, we  write ${\rm cl}_G\left(S\right)$ for its closure in $\overline X^G$,
and  ${\rm cl}_H\left(S\right)$ for its closure in $\overline X^H$.
When the space $(X,d)$ has approaching geodesics, then by Theorem \ref{abstract} the Gromov and horofunction compactifications are equivalent, so these distinctions are not needed.

\begin{proposition}[{\cite[Proposition 4.4]{WWpacific}}]
\label{prop:boundeddistance}
Let $(X,d)$ be a proper geodesic Gromov hyperbolic metric space and let $p\in X$. 
Given $a,b\in\partial_HX$, we have $\Phi(a)=\Phi(b)$ if and only if 
$$
\sup_{x\in X}|h_{a,p}(x)-h_{b,p}(x)|\le M,\qquad\text{ for some }M>0.
$$
Furthermore if the space is $\delta$-hyperbolic and $\Phi(a)=\Phi(b)$, then we can choose $M=2\delta$.
\end{proposition}

\begin{remark}\label{pennywise}
Notice that for all $a\in \partial_HX$, $z\in X$, $$d(p,z)\geq -h_{a,p}(z).$$
\end{remark}



\begin{proposition}
\label{prophor}
Let $(X,d)$ be a proper geodesic Gromov hyperbolic metric space and let $p\in X$. Then
\begin{enumerate}
\item For every $a\in\partial_HX$ and $R>0$, we have 
\begin{align*}
{\rm cl}_G\left(E_p(a,R)\right)\cap\partial_G X&=\{\Phi(a)\}.
\end{align*}
\item For every $a\in\partial_HX$, we have 
\begin{equation}
\label{intersectcomp}
\bigcap_{R>0}E_p(a,R)=\varnothing\qquad\text{and}\qquad \bigcap_{R>0}{\rm cl}_G\left(E_p(a,R)\right)=\{\Phi(a)\}.
\end{equation}
\item Let $a,b\in\partial_HX$ such that $\Phi(a)\neq\Phi(b)$. 
Then there exists $R>0$ such that ${\rm cl}_G\left(E_p(a,R)\right)$ and ${\rm cl}_G\left(E_p(b,R)\right)$ are disjoint.
\end{enumerate}
\end{proposition}
\begin{proof}


{(1)} Let $\gamma$ be a geodesic ray starting at $p$ and so that $\Phi(a)=[\gamma]$.
Let $b:=B(\gamma)\in \partial_HX$ be the image of the geodesic ray under the Busemann map defined in Section \ref{sec:background}. Notice that $\Phi(a)=\Phi(b)$.

Given $R>0$ the point $x_n:=\gamma(n)$ belongs to the horosphere $E_p(b,R)$ for every $n$ sufficiently large. 
Indeed we have that
$$
h_{b,p}(x_n)=\lim_{t\to\infty} d(\gamma(n),\gamma(t))-d(\gamma(t),p)=-n.
$$

By Proposition \ref{prop:boundeddistance} we conclude that $h_{a,p}(x_n)$ also converges to $-\infty$, and thus that $x_n\in E_p(a,R)$ for every $n$ sufficiently large.
This implies that $\Phi(a)\in {\rm cl}_G\left(E_p(a,R)\right)$.

Suppose that $(x_n)$ is a sequence in $E_p(a,R)$ that $G$-converges to a point in $\partial_G X$. 
Let $(w_n)$  be a sequence in $X$ that $H$-converges to $a$. 
For every $n\ge 0$ we may choose a positive integer $m_n\ge 0$ such that $d(x_n,w_{m_n})-d(w_{m_n},p)\le \log R$. 
Without loss of generality, we may further assume that $\displaystyle \lim_{n\to\infty}m_n=\infty$.

We find that
\begin{align*}
2(x_n|w_{m_n})_p&=d(x_n,p)+d(w_{m_n},p)-d(x_n,w_{m_n})\ge d(x_n,p)- \log R,
\end{align*}
and therefore $\displaystyle \lim_{n\to\infty} (x_n|w_{m_n})_p=+\infty$. 
The sequence $(w_n)$ $G$-converges to $\Phi(a)$, 
therefore by Remark \ref{infinitysequences} we obtain that the sequence $(x_n)$ also $G$-converges to $\Phi(a)$.
This proves that $\Phi(a)$ is the unique point in ${\rm cl}_G\left(E_p(a,R)\right)\cap\partial_GX$.

\par\medskip
{(2)} Suppose there exists $x\in\bigcap_{R>0}E_p(a,R)$. Then for every $R>0$ we have, by Remark \ref{pennywise},
$$
-d(x,p)\le h_{a,p}(x)<\log R,
$$
and therefore $d(x,p)>-\log R$, giving a contradiction.
Since the horofunction $h_{a,p}$ is continuous on $X$, given $0<R<R'$, one has $\{\Phi(a)\}\subseteq {\rm cl}_G\left(E_p(a,R)\right)\subseteq E_p(a,R')\cup\{\Phi(a)\}$, and therefore $\bigcap_{R>0}{\rm cl}_G\left(E_p(a,R)\right)=\{\Phi(a)\}$.

{(3)} The Gromov compactification $\overline X^G$ is metrizable, therefore we may choose a distance $\widetilde d$ inducing the topology of $\overline X^G$. 
We write $\widetilde B(x,\varepsilon)$ for the ball of center $x\in \overline X^G$ and radius $\varepsilon>0$, with respect to  $\widetilde d$.

By \eqref{intersectcomp} it follows that
\begin{align*}
\varnothing&=\{\Phi(a)\}\setminus \widetilde B(\Phi(a),\varepsilon)=\bigcap_{R>0} {\rm cl}_G\left(E_p(a,R)\right)\setminus \widetilde B(\Phi(a),\varepsilon).
\end{align*}

The last term on the right is the intersection of a nested family of compact sets. Therefore we may find $R>0$ so that ${\rm cl}_G\left(E_p(a,R)\right)\subseteq \widetilde B(\Phi(a),\varepsilon)$. The result follows by taking $\varepsilon$ sufficiently small so that the balls with center $\Phi(a)$ and $\Phi(b)$ and radius $\varepsilon$ are disjoint.
\end{proof}

We introduce the following definition.
\begin{definition}
Let $(X,d)$ be a proper geodesic Gromov hyperbolic metric space. 
Given $R>0$ and  a geodesic ray $\gamma\in \mathscr{R}(X)$, the \textit{geodesic region} $A(\gamma,R)$ is the open subset of $X$ of the form
$$A(\gamma, R):=\{x\in X\colon d(x,\gamma)<R\}.$$
The point $[\gamma]\in\partial_GX$ is called the \textit{vertex} of the geodesic region.

\end{definition}

\begin{remark}
If $D\subset \C^q$ is a bounded strongly convex domain with $C^3$ boundary, geodesic regions are comparable to Koranyi regions (see \cite[Section 7]{AAG}). 
\end{remark}

\begin{proposition}\label{propA}
Let $(X,d)$ be a proper geodesic Gromov hyperbolic metric space.
\begin{enumerate}
	\item  Two geodesic regions with same vertex are comparable.
	More precisely, given two geodesic rays $\gamma,\sigma\in \mathscr{R}(X)$ with $[\gamma]=[\sigma]$, we can find $R_0>0$ so that 
	$$A(\sigma,R)\subseteq A(\gamma,R+R_0),\quad A(\gamma,R)\subseteq A(\sigma,R+R_0),\qquad\forall R>0.$$
	\item Every geodesic region satisfies
	$$ {\rm cl}_G\left(A(\gamma,R)\right)\cap \partial_GX=\{[\gamma]\}.$$
	\item Let $(x_n)$ in $A(\gamma,R)$ be a sequence $G$-converging to $[\gamma]$. Then for every base point $p\in X$, we have
	$$\lim_{n\to\infty}B_\gamma(x_n,p)=-\infty.$$
\end{enumerate}
\end{proposition}
\proof
(1) The geodesic rays $\gamma$ and $\sigma$ are asymptotic, meaning that there exists $R_0>0$ such that
$$\sup_{t\geq0}d(\gamma(t),\sigma(t))<R_0.$$
 Let $R>0$. Given $x\in A(\sigma,R)$ we can choose $t^*\geq0$ so that $d(x,\sigma(t^*))<R$. Then
$$d(x,\gamma(t^*))\leq d(x,\sigma(t^*))+d(\sigma(t^*),\gamma(t^*))<R+R_0,$$
i.e. $x\in A(\gamma,R+R_0)$. 

(2) For every $R,t>0$ we have $\gamma(t)\in A(\gamma,R)$, showing that $[\gamma]\in {\rm cl}_G\left(A(\gamma,R)\right)\cap \partial_GX$. Conversely, let $(x_n)$ be a sequence in $A(\gamma,R)$ $G$-converging to $\xi\in\partial_GX$. For every $n\in \mathbb{N}$ let $t_n\geq0$ such that $d(x_n,\gamma(t_n))<R$. Fix a base point $p\in X$. Then one has
$$2(x_n|\gamma(t_n))_{p}=d(x_n,p)+d(\gamma(t_n),p)-d(x_n,\gamma(t_n))>d(x_n,p)+d(\gamma(t_n),p)-R\to+\infty.$$
It follows that both sequences $(x_n)$ and $(\gamma(t_n))$ $G$-converge to the same point, i.e. $\xi=[\gamma]$.

(3) For every $n\in\N$ let $t_n\geq0$ such that $d(x_n,\gamma(t_n))<R$.  For all $t\geq t_n$ we have
\begin{align*}d(x_n,\gamma(t))-d(\gamma(t),p)&\leq d(x_n,\gamma(t_n))+d(\gamma(t_n),\gamma(t))-d(\gamma(t),\gamma(0))+d(\gamma(0),p)\\&\leq R+d(\gamma(0),p)-t_n.
\end{align*}
Since $t_n\to+\infty$,  it follows that $B_\gamma(x_n,p)\to -\infty$.
\endproof

Geodesic regions allow us to generalize the classical $K$-limits from complex analysis to proper geodesic Gromov hyperbolic metric spaces.
\begin{definition}
Let $(X,d)$ be a proper geodesic Gromov hyperbolic metric space and let $Y$ be a Hausdorff topological space. Let $f\colon X\rightarrow Y$ a map and  let $\eta\in\partial_GX$, $\xi\in Y$. We say that $f$ has \textit{geodesic limit} $\xi$ at $\eta$
if for every sequence $(x_n)$ converging to $\eta$ contained in a geodesic region with vertex $\eta$, the sequence $(f(x_n))$ converges to $\xi$.
\end{definition}

\begin{definition}
Let $(X,d)$ be a proper geodesic Gromov hyperbolic metric space. 
 Let $f\colon X\to X$ be a non-expanding self-map. 
Given $\eta\in \partial_GX$, we define the \textit{dilation} of $f$ at $\eta$ with respect to the  base point $p\in X$ as the (possibly infinite) number $\lambda_{\eta,p}>0$ such that
$$\log\lambda_{\eta,p}=\liminf_{z\to \eta} d(z,p)-d(f(z),p).$$ 
\end{definition}
For the definition of the dilation  of a holomorphic self-map at a boundary point, see  e.g.\cite{AbBook}.
\begin{remark}
Notice that 
\begin{equation}\label{cyril}
d(p,z)-d(p,f(z))\geq  d(f(p),f(z))-d(p,f(z))\geq -d(p,f(p)).
\end{equation}
In particular, $\log\lambda_{\eta,p}>-\infty.$
\end{remark}
\begin{lemma}
Let $(X,d)$ be a proper geodesic Gromov hyperbolic metric space. 
 Let $f\colon X\to X$ be a non-expanding self-map and let $\eta\in \partial_GX$. Assume that there exists $p\in X$ such that
 $\log\lambda_{\eta,p}<+\infty$. Then for all $q\in X$,
 $$\log\lambda_{\eta,q}\leq \log\lambda_{\eta,p}+2d(p,q).$$
\end{lemma}
\begin{proof}
Let $(x_n)$ be a sequence converging to $\eta$ such that $ d(x_n,p)-d(f(x_n),p)\underset{n \to \infty}{\longrightarrow} \log\lambda_{\eta,p}$.\
We have 
$$d(q,x_n)-d(q,f(x_n))\leq d(q,p)+d(p,x_n)-d(p, f(x_n))+d(p,q).$$
The result follows by letting $n\to \infty$.
\end{proof}

\begin{remark}
It follows from the previous lemma that the condition $\lambda_{\eta,p}<+\infty$ is independent on the choice of the base point $p\in X$.
\end{remark}

\begin{lemma}\label{weakJ}
Let $(X,d)$ be a proper metric space and let $f\colon X\to X$ a non-expanding map. Assume that there exists a sequence $(w_n)$ in $X$ such that 
\begin{enumerate}
\item $w_n\underset{n \to \infty}{\longrightarrow} a\in \partial_HX$,
\item $f(w_n)\underset{n \to \infty}{\longrightarrow} b\in \partial_HX$, and
\item $d(p, w_n)-d(p, f(w_n))\underset{n \to \infty}{\longrightarrow} A<+\infty.$
\end{enumerate}
Then
\begin{equation}\label{horofunctionjulia}
h_{b,p}\circ f\leq h_{a,p}+A.
\end{equation}
\end{lemma}
\begin{proof}
For all $z\in X$,  
\begin{align*}
d(f(z),f(w_n))-d(f(w_n),p)&\leq d(z,w_n)-d(f(w_n),p)\\
&\leq  d(z,w_n)-d(w_n,p) +d(w_n,p) -d(f(w_n),p).
\end{align*}
Taking the limit as $n\to \infty$ on both sides we obtain \eqref{horofunctionjulia}.
\end{proof}

\begin{proposition}\label{strangerthings}
Let $(X,d)$ be a proper geodesic Gromov hyperbolic metric space. 
 Let $f\colon X\to X$ be a non-expanding self-map. Let $p\in X$ and   $\eta\in \partial_GX$ be such that $\lambda_{\eta,p}<+\infty$. Then there exists $\xi\in \partial_GX$ such that  $f$ has geodesic limit $\xi$ at $\eta$.
 \end{proposition}
 \begin{proof}
Let $(w_n)$ be a sequence $G$-converging to $\eta$ such that $d(w_n,p)-d(f(w_n),p)\underset{n\to \infty}{\longrightarrow} \log \lambda_{\eta,p}$. The $G$-convergence of $(w_n)$ implies that $d(w_n,p)\underset{n \to \infty}{\longrightarrow} +\infty$. Since $\log \lambda_{\eta,p}$ is finite, we obtain $d(f(w_n),p)\underset{n \to \infty}{\longrightarrow}+\infty$. Consequently, up to extracting a subsequence we may assume that $(w_n)$ $H$-converges to a point $a\in \Phi^{-1}(\left\lbrace\eta\right\rbrace )$ and $(f(w_n))$ $H$-converges to a point $b \in \partial_HX$. By Lemma \ref{weakJ} we obtain
$$h_{b,p}\circ f\leq h_{a,p}+\log\lambda_{\eta,p}.$$
Let $\gamma$ be a geodesic ray such that $[\gamma]=\eta$, and consider the horofunction $B_\gamma\left(\cdot,p\right)$, where $B_\gamma$ is the Busemann function associated to $\gamma$. 
By Proposition \ref{prop:boundeddistance}, there exists a constant $M>0$ such that 
$$|h_{a,p}(x)-B_\gamma(x,p)|\leq M,\quad\forall x\in X.$$ Hence
$$h_{b,p}\circ f\leq B_\gamma(\cdot,p)+M+\log\lambda_{\eta,p}.$$
Let $(x_n)$ be a sequence $G$-converging to $\eta$ inside a geodesic region. By point (3) of Proposition \ref{propA} we have that
$B_\gamma(x_n,p)\to-\infty$, hence  $h_{b,p}(f(x_n))\to-\infty.$ By Remark \ref{pennywise}  we have that $d(f(x_n),p)\to+\infty.$
It follows from Proposition \ref{prophor} that the sequence $(f(x_n))$ $G$-converges to $\xi:=\Phi(b)$.
 \end{proof}
 
 The proof of Proposition \ref{strangerthings} combined with the uniqueness of the geodesic limit immediately yields the following result.
 \begin{proposition}\label{unicitalimitegeodetico}
 Let $(X,d)$ be a proper geodesic Gromov hyperbolic metric space. 
 Let $f\colon X\to X$ be a non-expanding self-map. Let $p\in X$ and   $\eta \in \partial_GX$ such that $\lambda_{\eta,p}<+\infty$. Let $\xi$ the geodesic limit of $f$ at $\eta$. Then if 
 $(x_n)$ is a sequence in $X$ $G$-converging to $\eta$ such that
the sequence $d(x_n,p)-d(f(x_n),p)$ is bounded from above,
then $f(x_n)$ $G$-converges to $\xi$.
 \end{proposition}

\begin{definition}
Let $(X,d)$ be a proper geodesic Gromov hyperbolic metric space. Let $f\colon X\to X$ be a non-expanding map. We say that a point $\eta\in \partial_GX$ is  a \textit{boundary regular fixed point} (BRFP for short) if
$\lambda_{\eta,p}<+\infty$ and  if $f$ has geodesic limit $\eta$ at $\eta$. 
\end{definition}
 For the definition of boundary regular fixed points in bounded strongly convex domains, see e.g. \cite{AbRaback}.

We now discuss the dynamics of non-expanding self-maps in  proper geodesic Gromov hyperbolic metric spaces. 
We first recall some definitions.
\begin{definition}
Let $(X,d)$ be a metric space and $f\colon X\to X$ be a non-expanding self-map. 
\begin{itemize}
\item  A sequence $(f^n(x))_{n\in \N}$ with $x\in X$ is called a \textit{forward orbit} of $f$.
\item A sequence $(x_n)_{n\in \N}$ is called a \textit{backward orbit} if $f(x_{n+1})=x_n$ for every $n\in\N$.
\item A sequence $(x_n)_{n\in \mathbb{Z}}$ is called a \textit{complete orbit} if $f(x_{n})=x_{n+1}$ for every $n\in \mathbb{Z}$.
	\item Let $x\in X$ and $m\geq1$, the \textit{forward m-step} $s_m(x)$ of $f$ at $x$ is the limit 
	$$s_m(x):=\lim_{n\to \infty}d(f^n(x),f^{n+m}(x)) .$$
	Such limit exists since the sequence $(d(f^n(x),f^{n+m}(x)))_{n\geq0}$ is non-increasing.
	
	\item Let $x\in X$, the \textit{divergence rate} (or \textit{translation lenght}, or \textit{escape rate}) $c(f)$ of $f$ is the limit
	$$c(f):=\lim_{n\to\infty}\frac{d(x,f^n(x))}{n}. $$
	The limit exists thanks to the subadditivity of the sequence $d(x,f^n(x))$, and the definition  does not depend on the choice of $x\in X$. By \cite[Proposition 2.7]{AB} we have that $$c(f)=\lim_{m\to+\infty}\frac{s_m(z_0)}{m}=\inf_{m}\frac{s_m(z_0)}{m}.$$
	\item Let $\beta:=(x_n)_{n\geq0}$ be a backward orbit, the \textit{backward  m-step} $\sigma_m(\beta)$ of $f$ at $\beta$ is the limit
	$$\sigma_m(\beta):=\lim_{n\to\infty}d(x_n,x_{n+m}).$$
    Such limit exists since $(d(x_n,x_{n+m}))_{n\geq0}$ is non-decreasing. We say that $\beta$ has \textit{bounded step} if $\sigma_1(\beta)<\infty$. By subadditivity of the sequence $(\sigma_m(\beta))_m$ the limit
    $$b(\beta):=\lim_{m\to+\infty}\frac{\sigma_m(\beta)}{m}$$ exists and equals $\inf_m\sigma_m(\beta)/m$. We call the number $b(\beta)$  the \textit{backward step rate} of $(\beta).$	
\end{itemize}
\end{definition}

The next result establishes  relations between the forward and backward dynamics of a non-expanding map $f$ and its boundary regular fixed points.
\begin{proposition}\label{king}
Let $(X,d)$ be a proper geodesic Gromov hyperbolic metric space. Let $f\colon X\to X$ be a non-expanding self-map.
\begin{enumerate}
\item Assume that $f$ has  escaping forward orbits, and let $\zeta\in \partial_GX$ be its Denjoy--Wolff point. Then $\zeta$ is a BRFP and $$\log\lambda_{p,\zeta}\leq -c(f).$$
\item
If a point  $\eta\in \partial_GX$ is a limit point of a  backward orbit with bounded step $(w_n)$, then $\eta$ is a BRFP
and  $$\log\lambda_{p,\eta}\leq b(w_n).$$
\end{enumerate}
\end{proposition}
\begin{proof}
Fix a point $z_0\in X$. Since $d(z_0, f^{n}(z_0) )\to +\infty$, there  exists a subsequence of iterates $(f^{n_k})$ such that for all $k\geq 0$,
$$d(z_0, f^{n_k}(z_0) )<d(z_0, f^{n_k+1}(z_0) ).$$ Setting $w_k:= f^{n_k}(z_0)$ we have  $w_k\to \zeta$, $f(w_k)\to \zeta$ and for all $k\geq 0$,
$$d(z_0, w_k)-d(z_0, f(w_k))< 0.$$
By Proposition \ref{unicitalimitegeodetico} we obtain that
 $\zeta$ is a BRFP.  Moreover
\begin{align*}
-c(f)&=\lim_{n\to +\infty}\frac{-d(z_0,f^n(z_0))}{n}\\&\geq \liminf_{n\to+\infty}[-d(z_0,f^{n+1}(z_0))+d(z_0,f^n(z_0))]\\&\geq \liminf_{z\to \zeta}[d(z_0,z)-d(z_0, f(z))], 
\end{align*}
and this proves (1). 

To prove (2), assume that $(w_{n_k})$ is a subsequence of $(w_n)$ converging to $\eta$. The sequence
$f(w_{n_k})=w_{n_k-1}$ also converges to $\eta$ since  $d(w_{n_k},w_{n_k-1}) \leq \sigma_1(w_n)$ .
Moreover, 
$$ d(p,w_{n_k})-d(p,w_{n_k-1})\leq d(w_{n_k},w_{n_k-1})    \leq \sigma_1(w_n),$$
thus it follows from Proposition \ref{unicitalimitegeodetico} that
 $\eta$ is a BRFP.
 
  Moreover, for all $m\geq 1$,
\begin{align*}
\sigma_m(w_n)&=\lim_{n\to+\infty}d(w_{n},w_{n-m})\geq \liminf_{k\to+\infty}d(p,w_{n_k})-d(p,w_{n_k-m})\\
&= \liminf_{k\to+\infty}d(p,w_{n_k})-d(p,w_{n_k-1})+\dots +d(p,w_{n_k-m+1})-d(p,w_{n_k-m})\geq  m\log\lambda_{\eta,p}.
\end{align*}

\end{proof}


\begin{example}
In general  the dilation at a BRFP may depend on the base point.
Let $(X,d)$ be as in Example \ref{wwexample}. Consider the non-expanding self-map $f:X\to X$ defined as
$$f(a,b)=(a+1,1),\ \ \ (a,b)\in X. $$ Then $c(f)=1$, and the Denjoy--Wolff point of $f$ is 
the unique point in $\partial_GX$, denoted $\eta$.
By Proposition \ref{king} $\eta$ is a BRFP and for all $p\in X$
we have $\log\lambda_{\eta,p}\leq -c(f),$
but the value  $\lambda_{\eta,p}$   depends on the base point $p$. Indeed,
$$\log\lambda_{\eta,(0,1)}=\liminf_{(a,b)\to\eta}[a+1-b]-[a+1]=-1$$
and 
$$\log\lambda_{\eta,(0,-1)}=\liminf_{(a,b)\to\eta}[a+b+1]-[a+1+2]=-3.$$

%
\end{example}

We now show that under some assumption backward and forward orbits are discrete quasi-geodesics, and thus converge to a BRFP inside a geodesic region. We start by recalling the definition of discrete quasi-geodesics.
\begin{definition}
Fix $A\geq1$, $B\geq0$. 
A sequence $(x_n)_{n\geq0}$ is a \textit{discrete} $(A,B)$-\textit{quasi-geodesic ray} if for every $n,m\geq0$
$$A^{-1}|n-m|-B\leq d(x_n,x_m)\leq A|n-m|+B.$$
Similarly one can define discrete quasi-geodesics lines $(x_n)_{n\in \mathbb{Z}}.$
\end{definition}
 \begin{remark}
	\label{lem_Interpol}
	If $\left(x_n\right)_{n\geq0}$ is a discrete $\left(A,B\right)$-quasi-geodesic ray, then $\gamma:\R_{\geq0}\to X$ given by
	$$\gamma(t):=x_{\lfloor t\rfloor}, \ \ t\geq0, $$
	where $\lfloor\cdot\rfloor$ is the floor function, is a $(A,A+B)$-quasi-geodesic ray with $\gamma(n)=x_n$ for every $n\in\N$.
	Analogously, starting with  a discrete  $\left(A,B\right)$-quasi-geodesic line $\left(x_n\right)_{n\in \mathbb{Z}}$ we obtain  a $(A,A+B)$-quasi-geodesic  line $\gamma:\R\to X$.
\end{remark}

We introduce the following definition:
\begin{definition}
Let $(X,d)$ be a   proper geodesic Gromov hyperbolic metric space and $f\colon X\to X$ be a non-expanding self-map. 
We call $f$ \textit{elliptic} if all forwards orbits are bounded. 
If $f$ is not elliptic, then by Remark \ref{Rem_dicho}, every forward orbit escapes.
In this case we say that $f$ is parabolic if $c(f)=0$, and we say that $f$ is hyperbolic if $c(f)>0.$
\end{definition}

This definition generalizes at the same time the classification of holomorphic self-maps of a bounded strongly convex domain of $\C^n$ (see e.g. \cite{AbRaback}), and the  classification of isometries of a proper geodesic Gromov hyperbolic metric space (see for instance \cite{CDP}). Recall that every isometry $f\colon X\to X$ of a proper geodesic Gromov  hyperbolic metric space extends to a homeomorphism of the Gromov compactification, and that if $f$ is non-elliptic, then $f$ has either one or two fixed points at the boundary.
\begin{theorem}[Classification of isometries, see e.g. \cite{CDP}]\label{wellknown}
Let $f$ be a non-elliptic isometry of a  proper geodesic Gromov  hyperbolic metric space $X$. The following are equivalent:
\begin{enumerate}
\item $f$ is hyperbolic,
\item every complete orbit of $f$ is a discrete quasi-geodesic line,
\item there exists a complete orbit of $f$ which is a discrete quasi-geodesic line,
\item the continuous extension of $f$ to the Gromov compactification has two fixed point at the boundary.
\end{enumerate}
\end{theorem}

The following result generalizes several results in complex dynamics to our general setting, see for instance \cite[Section 3.5]{BraPog}, \cite[Theorem 1.8]{O}, \cite[Theorem 1 (iii)]{errata}, and \cite[Lemma 8.9]{AAG}. 
Such results are classically obtained with careful estimates of the Kobayashi metric near the boundary of the domain. 
The method of proof presented here is new and based on the Gromov shadowing lemma. Point (1) answers  the question in  \cite[Remark 4.4]{AAG}. 
Point (2) answers affirmatively the conjecture in \cite[Remark 2]{errata}.
\begin{proposition}\label{entersgeodes}
Let $(X,d)$ be a proper geodesic Gromov hyperbolic metric space. Let $f\colon X\to X$ be a non-expanding self-map.   
\begin{enumerate}
\item If $f$ is hyperbolic, then every forward orbit is a discrete quasi-geodesic ray and is contained in a geodesic region with vertex the Denjoy--Wolff point $\zeta$.
\item  If $f$ is hyperbolic, then  every backward orbit with bounded step is a discrete quasi-geodesic ray converging to a  BRFP $\eta\in \partial_GX$ different from the Denjoy--Wolff point $\zeta$, and is  contained in a geodesic region with vertex   $\eta$.
\item Let $(z_n)$ be a backward orbit with bounded step converging to a BRFP $\eta$, and assume that there exists a base point $p\in X$ such that  $\lambda_{\eta,p}>1$. Then $(z_n)$ is  a discrete quasi-geodesic ray, and is  contained in a geodesic region with vertex   $\eta$.
\end{enumerate}
\end{proposition}
\proof
(1) A discrete quasi-geodesic ray is contained in  a geodesic region: indeed by Remark \ref{lem_Interpol} it can be interpolated by a quasi-geodesic ray $\gamma$, hence by the Shadowing Lemma \cite[Chapter 5]{GdlH} there exists a geodesic ray $\sigma$ and $R>0$ such that $$\gamma(\R_{\geq0})\subseteq A(\sigma,R).$$

We show that  every  forward orbit is a discrete quasi-geodesic ray. Let $(x_n)$ be a forward orbit, then for every $m,n\geq 0$
$$d(x_n,x_m)\leq|m-n|d(x_0,x_1)$$
and
$$d(x_n,x_m)\geq s_{|m-n|}(x_0)\geq c(f)|m-n|,$$
so $(x_n)$ is a discrete $(A,0)$-quasi-geodesic with $A:=\max\{c(f)^{-1},d(x_0,x_1)\}$. This proves (1). 

(2) Let $(z_n)$ be a backward orbit  with bounded step, and consider the complete orbit defined by 
$$x_n:=\begin{cases}
z_{-n} & n\leq 0 \\
f^n(z_0)& n\geq 0.\\
\end{cases}   $$
The complete orbit $(x_n)_{n\in \mathbb{Z}}$ is a  discrete quasi-geodesic line. Indeed for all $m, n\in \mathbb{Z}$ we have that 
$$      \sigma_1(z_n)|m-n|   \geq d(x_n,x_m)\geq  c(f)|m-n|.$$ By the  Shadowing Lemma  there exists a geodesic line $\sigma$ and $R>0$ such that $$x_n\in A(\sigma,R), \quad \forall n\in \mathbb{Z}.$$ Notice that a geodesic line in a Gromov space has different endpoints at $\pm\infty$. It follows that  $(z_n)$ converges to a point $\eta$ different from the Denjoy--Wolff point $\zeta$, and that $(z_n)$ is  contained in a geodesic region with vertex $\eta$.  

(3) We have  $d(x_n,x_m)\leq\sigma_1(z_n)|m-n| $ for all $m,n\geq 0$. To obtain the other inequality,  let $0<a<\log\lambda_{\eta,p}$. There exists $N>0$ such that for all $n\geq N$ we have that $d(p,z_{n+1})-d(p,z_n)\geq a.$ Then  we have that 
\begin{equation}\label{tenet}
d(z_m,z_n)\geq a|m-n|,\quad \forall m,n\geq N.
\end{equation} 

\endproof

The following result is an immediate consequence of point (1) of the previous proposition.
\begin{corollary}
Let $(X,d)$ be a proper geodesic Gromov hyperbolic metric space. Let $f\colon X\to X$ be a non-elliptic non-expanding self-map.  Then the following are equivalent:
\begin{enumerate}
\item  $f$ is hyperbolic,
\item every forward orbit is a discrete quasi-geodesic ray,
\item there exists a forward orbit which is a discrete quasi-geodesic ray.
\end{enumerate}
\end{corollary}
Notice that there are examples of  parabolic non-expanding maps  such that every forward orbit is contained in a geodesic region with vertex the Denjoy--Wolff point $\eta$. A classical example is the  holomorphic self-map $z\mapsto z+1$ in the right half plane $\H$. However, for isometries  this cannot happen, as the following result shows.
\begin{proposition}
Let $(X,d)$ be a proper geodesic Gromov hyperbolic metric space such that $\partial_GX$ contains at least two points. Let $f\colon X\to X$ be a non-elliptic isometry and let $\eta$ be its Denjoy--Wolff point.
The following are equivalent:
	\begin{enumerate}
		\item $f$ is hyperbolic,
		\item every forward orbit of $f$ is contained in a geodesic region with vertex $\eta$,
		\item there exist a forward orbit of $f$ which is contained in a geodesic region with vertex $\eta$,
		\item $f$ and $f^{-1}$ have different Denjoy--Wolff points.
	\end{enumerate}
\end{proposition}

\begin{proof}
$[(1)\Rightarrow(2)]$ is Proposition \ref{entersgeodes}.

$[(3)\Rightarrow(4)]$
Notice that if a forward orbit of $f$ is contained in a geodesic region with vertex the Denjoy--Wolff point $\eta$, then every forward orbit is. Since $\partial_GX$ contains at least two points, there exists a geodesic line  $\gamma$  such that $\lim_{\to+\infty}\gamma(t)=\eta$.
Let $x_0:=\gamma(0)$ and denote $x_n:=f^n(x_0)$ for all $n\in \mathbb{Z}$. It is enough to prove that 
$$\limsup_{n+\infty}(x_n|x_{-n})_{x_0}<+\infty.$$
Notice that
$$2(x_n|x_{-n})_{x_0}:=d(x_n,x_0)+d(x_{-n},x_0)-d(x_{-n},x_n)=d(x_0,x_n)+d(x_n,x_{2n})-d(x_{2n},x_0).$$
By assumption there exits $R>0$ such that $x_n\in A(\gamma|_{\R_{\geq 0}},R)$. Hence for every $n\in\mathbb{N}$ there exits $y_n\in\gamma$ such that $d(x_n,y_n)\leq R$. Finally, using Proposition 7.4 in \cite{AAG} we have
\begin{align*}
	&d(x_0,x_n)+d(x_{n},x_{2n})-d(x_{2n},x_0)\\
	&\leq  2d(x_n,y_n)+d(x_0,y_n)+d(y_n,y_{2n})+d(y_{2n},x_{2n})-d(x_0,y_{2n})-d(y_{2n},x_{2n})+6\delta\\
	&\leq4R+6\delta+d(x_0,y_n)+d(y_n,y_{2n})-d(x_0,y_{2n})\leq 4R+6\delta.
\end{align*}
 $[(4)\Rightarrow(1)]$ This is well-know, and follows immediately from Theorem \ref{wellknown}. We give a short self-contained proof for completeness. Fix $x_0\in X$ and denote $x_n=f^n(x_0)$ with $n\in\mathbb{Z}$. Since $x_n$ and $x_{-n}$ converge to two different points in $\partial_GX$ there exits $R>0$ such that, for every $n,m\in\N$,
\begin{equation}\label{gpr1}
2(x_n|x_{-m})_{x_0}=d(x_n,x_0)+d(x_m,x_0)-d(x_{n+m},x_0)<R.
\end{equation}

Since $d(x_n,x_0)\to+\infty$, there exists $N\in\N$ such that $d(x_N,x_0)>R$. Now from (\ref{gpr1}) we have, for all $n\geq 1$,
$$d(x_{nN},x_0)>d(x_{(n-1)N},x_0)+d(x_N,x_0)-R ,$$
consequently 
$$d(x_{nN},x_0)>n[d(x_N,x_0)-R],$$
which implies that $c(f^N)>0$. We conclude noticing that $c(f^N)=Nc(f)$.
\end{proof}

In the remaining part of the section, we refine our results adding the assumption that  $\overline X^H$ and $\overline X^G$ are topologically equivalent, which by Theorem \ref{weaknotion} is the case if $X$ satisfies the weak    approaching geodesics property.   In this case we can identify the two compactifications, hence it makes sense to consider a horosphere $E_p(\eta,R)$ centered at a point $\zeta$ in the Gromov boundary. We start with a generalization  of the classical Julia's Lemma.

\begin{theorem}[Julia's Lemma]\label{Julia}
Let $(X,d)$ be a proper geodesic Gromov hyperbolic metric space such that $\overline X^H$ is topologically equivalent to $\overline X^G$. Let $f\colon X\to X$ be a non-expanding self-map.  Let $p\in X$ and   $\eta \in \partial_GX$ such that $\lambda_{\eta,p}<+\infty$.  Let $\xi\in \partial_GX$ be the geodesic limit of $f$  at $\eta$.
Then,  for all $R>0$,
\begin{equation}\label{carrie}
f(E_p(\eta,R))\subseteq E_p(\xi,\lambda_{\eta,p} R).
\end{equation}
\end{theorem}
\begin{proof}
Let $(w_n)$ be a sequence converging to $\eta$ such that $d(w_n,p)-d(f(w_n),p)\to \log \lambda_{\eta,p}.$
Since $X$ is complete, $d(p,w_n)\to +\infty$, and thus  $d(f(w_n),p)\to +\infty$. Up to extracting a subsequence we can hence assume that 
$(w_n)$ converges to a point $\xi \in \partial_GX$. By Lemma \ref{weakJ}, for all $R>0$ we have \eqref{carrie}.
By Proposition \ref{unicitalimitegeodetico} it follows that $f$ has geodesic limit $\xi$ at $\eta$.
\end{proof}

\begin{remark}
Abate generalized Julia's Lemma to bounded strongly convex domains of $\C^q$ with $C^3$ boundary (see \cite[Proposition 2.7.15]{AbBook}). It follows from the previous theorem  that the same statement holds for bounded  strongly pseudoconvex domains with $C^2$ boundary  and for bounded convex domains of D'Angelo finite type with $C^\infty$ boundary.
\end{remark}

 We now show some applications of the Julia's Lemma. The first is that the dilation of a BRFP does not depend on the base point.
\begin{proposition}
Let $(X,d)$ be a proper geodesic Gromov hyperbolic metric space such that $\overline X^H$ is topologically equivalent to $\overline X^G$. Let $f\colon X\to X$ be a non-expanding map. Let $\eta\in \partial_GX$ and $p\in X$, and assume that $\lambda_{\eta,p}<\infty$. Let  $\xi$ be the geodesic limit of $f$ at $\eta$. Then
$$
\lambda_{\eta,p}=\min\{C>0 \,\colon\, f(E_p(\eta,R))\subseteq  E_p(\xi,CR) \quad \forall R>0\}.
$$
Furthermore if $\xi=\eta$ (that is, if $\eta$ is a BRFP),  then  the dilation $\lambda_{\eta,p}$ does not depend on the base point $p$.
\end{proposition}
\begin{proof}
Define 
$$\Gamma(p):=\{C>0 \,\colon\, f(E_p(\eta,R))\subseteq  E_p(\eta,CR)\quad\forall R>0\}.$$
Let $\gamma$ be a geodesic ray connecting $p$ to $\eta$. Fix $0<R<1$ and let $z_R\in \gamma$ such that $d(p,z_R)=-\log R$. Clearly $ z_R\in \partial E_p(\eta,R)$, and thus for every $C\in\Gamma(p)$ we have that
$f(z_R)\in {\rm cl}_G\left( E_p(\xi,CR)\right).$
It follows that $$d(p,f(z_R))\geq -\log (CR),$$
and thus
$d(p,z_R)-d(p,f(z_R))\leq \log C$. Since as $R\to 0$ we have that $z_R\to \zeta$.
It follows that
$$
\log\lambda_{\eta,p}=\liminf_{z\to \zeta}d(p,z)-d(p,f(z))\leq \log C\qquad\forall C\in\Gamma(p).
$$
By Theorem \ref{Julia} we have $\lambda_{\eta,p}\in\Gamma(p)$, therefore $\lambda_{\eta,p}=\min\Gamma(p)$.

If $\eta=\xi$ and we change the base point to $q\neq p$, then there exists a positive number $a>0$ so that $E_q(\eta,R)=E_p(\eta,aR)$ for every $R>0$. It follows that $\Gamma(p)=\Gamma(q)$, and therefore $\lambda_{\eta,p}=\lambda_{\eta,q}$.
\end{proof}

In what follows we will denote the dilation at a BRFP  $\eta$ simply by $\lambda_\eta$. The fact
that the dilation is independent on the base point allows us to partition the set of BRFPs in three classes:
\begin{definition}
Let $(X,d)$ be a proper geodesic Gromov hyperbolic metric space such that $\overline X^H$ is topologically equivalent to $\overline X^G$. Let $f\colon X\to X$ be a non-expanding self-map and let $\eta$ be a BRFP. We say
that $\eta$ is \textit{attracting} if $\lambda_\eta<1$, it is \textit{indifferent} if $\lambda_\eta=1$, and it is
 \textit{repelling} if $\lambda_\eta>1$.
\end{definition}

A second application of the Julia lemma is the following result relating forward dynamics and BRFPs.
\begin{theorem}\label{Korhyp}
Let $(X,d)$ be a proper geodesic Gromov hyperbolic metric space such that $\overline X^H$ is topologically equivalent to $\overline X^G$. Let $f\colon X\to X$ be a non-elliptic non-expanding self-map.
Let $\zeta\in \partial_GX$ be its Denjoy--Wolff point. Then $\zeta$ is the unique  attracting or indifferent BRFP of $f$. Moreover
\begin{equation}\label{divergencerate}
c(f)=-\log\lambda_\zeta,
\end{equation}
 and thus $f$ is hyperbolic iff its Denjoy--Wolff point is attracting, and is parabolic if and only if its Denjoy--Wolff point is indifferent.
\end{theorem}
\begin{proof}

By contradiction, let $\eta\in\partial_GX$ be a BRFP with $\lambda_{\eta}\leq1$ different from $\zeta$. Since $\eta$ is a boundary regular fixed point, by Theorem \ref{Julia} for every $R>0$ we have
$$f(E_p(\eta,R))\subseteq E_p(\eta,R).$$
Let $R>0$ such that $$K:=cl_G(E_p(\eta,R))\cap cl_G(E_p(\xi,R))\neq\varnothing.$$ Using Proposition \ref{prophor} one notices that $K$ is a compact set of $(X,d)$. Finally by Theorem \ref{Julia} one has $f(K)\subseteq K$, and so $f$ has no escaping forward orbits.

We are left with proving \eqref{divergencerate}. If $\lambda_\zeta=1$ the result follows immediately. Assume that $0<\lambda_\zeta<1$, and let $z\in E_p(\zeta,1)$.  Then it follows from Theorem \ref{Julia} that
$f^n(z)\in E_p(\zeta, \lambda^n_\zeta).$ Hence
$$\frac{d(p,f^n(z))}{n}\geq -\frac{\log \lambda_\zeta^n}{n}=-\log\lambda_\zeta.$$
\end{proof}

Both conclusions of the previous result are false without the assumption  that  $\overline X^H$ is   topologically equivalent to $\overline X^G$, even if we assume that the dilation does not depend on the base point, as the following example shows.
\begin{example}
Consider $(\R,d)$ and $(\mathbb{S}^1,d')$ where $d$ is the Euclidean  distance and  $d'$ is the inner distance induced by the Euclidean distance. Endow $X:=\R\times \mathbb{S}^1$ with the distance
$d''((x_1,y_1),(x_2,y_2))=d(x_1,x_2)+d(y_1,y_2).$
Then   $X$ is a proper geodesic Gromov hyperbolic space.
The Gromov boundary consists of two points $+\infty$ and $-\infty$, while the horofunction boundary is the disjoint union of two $\mathbb{S}^1$. Let $\vartheta\in [0,\pi]$ and let  $R_\theta\colon \mathbb{S}^1\to \mathbb{S}^1$ be the counterclockwise rotation by the angle $\vartheta$. The hyperbolic isometry $f\colon X\to X$ defined by $$f(x,y)=(x+1,R_\vartheta(y))$$ has divergence rate $c(f)=1$ and Denjoy--Wolff point $+\infty$, while $-\infty$ is the Denjoy--Wolff point of $f^{-1}.$ 
It is easy to see that  the dilations $\lambda_{-\infty,p}$ and $\lambda_{+ \infty,p}$ do not depend on the base point $p$ and are equal respectively to $e^{1-\vartheta}$ and $e^{-1-\vartheta}$. Hence if   $\vartheta>1$ we have that both dilations are strictly smaller than one. Notice that this shows that  the dilation is not the right tool to define attracting/indifferent/repelling BRFPs when the two compactifications are not equivalent.

\end{example}

We now move to backward dynamics and its relation  with BRFPs.
\begin{proposition}
Let $(X,d)$ be a proper geodesic Gromov hyperbolic metric space such that $\overline X^H$ is topologically equivalent to $\overline X^G$. Let $f\colon X\to X$ be a non-expanding map and let $(w_k)$ be a backward orbit with bounded step. Then,
\begin{enumerate}
\item  if   $\eta\in \partial_GX$ is a limit point of $(w_k)$, then $\eta$ is a  repelling or indifferent BRFP with dilation satisfying $\log\lambda_\eta\leq b(w_k)$,
\item if $f$ is hyperbolic, then there exists a repelling  BRFP $\eta$  such that $(w_k)\to \eta$ and   $(w_k)$ is a discrete quasi-geodesics contained  in a geodesic region with vertex $\eta$.
\item if there exists a  repelling  BRFP $\eta$   such that $(w_k)\to \eta$, then  $(w_k)$ is a discrete quasi-geodesics contained  in a geodesic region with vertex $\eta$.
\end{enumerate}
\end{proposition}

\proof
(1)
By Proposition \ref{king}, $\eta$ is a BRFP and  $\log\lambda_\eta\leq b(w_k)$. Assume by contradiction that $\lambda_\eta<1$, then by Theorem \ref{Julia} $f$ is non-elliptic and its Denjoy-Wolff point is $\eta$. Finally using (\ref{divergencerate}), we have $c(f)>0$ and so by (2) in Proposition \ref{entersgeodes} the backward orbit $(w_k)$ converges to a point different from $\eta$, which is a contradiction.

The remaining parts follow from Proposition \ref{entersgeodes} and Proposition \ref{Korhyp}.
\endproof

We conclude with some  open questions about backward dynamics. Assume $(X,d)$ is a proper geodesic Gromov hyperbolic metric space such that $\overline X^H$ is topologically equivalent to $\overline X^G$. Let $f\colon X\to X$ be a  non-expanding map.

\begin{enumerate}
\item If $f$ is parabolic, does every backward orbit with bounded step converge to a BRFP? If $f$ is elliptic,  does every escaping backward orbit with bounded step converge to a BRFP?
\item If $\zeta$ is a repelling BRFP, does there exist a backward orbit with bounded step converging to it? (If $\zeta$ is indifferent, this is not always the case, e.g. the holomorphic map $z\mapsto z+1$ on the right half-plane $\H$). 
\item If a backward orbit with bounded step $(w_k)$ converges to a BRFP $\eta$, is it true that 
$$\log\lambda_\eta=b(w_k)?$$
\end{enumerate}
Question (1) is open for holomorphic self-maps of the ball, and the answer is positive in the  disc \cite[Lemma 2.7]{Br}. 
It is proved in \cite{AAG} that for holomorphic self-maps of bounded strongly convex domains of $\C^q$ the answer to (2) is positive, and  if $\zeta$ is repelling the answer to (3) is positive. See \cite{PoCo1, O, AbRaback, AG} for previous  results related to Question (2).

\end{document}